\documentclass[11pt]{amsart}
\usepackage[active]{srcltx}
\usepackage{mathrsfs}
\usepackage[T1]{fontenc}

\usepackage{latexsym,enumerate}

\usepackage{amsmath,amssymb,amsthm,amsfonts,latexsym}
\usepackage[bookmarks]{hyperref}

\hypersetup{backref, colorlinks=true, colorlinks   = true,
urlcolor     = blue, linkcolor = blue, citecolor   = magenta
}

\def\cal{\mathcal}

\let\=\partial

\newtheorem {pro}{Proposition}[section]
\newtheorem {thm}[pro]{Theorem}
\newtheorem {cor}[pro]{Corollary}
\newtheorem{lem}[pro]{Lemma}

\theoremstyle{definition}
 \newtheorem {rem}[pro]{Remark}
\newtheorem {dfn}[pro]{Definition}

\newtheorem {exa}[pro]{Example}

\newtheorem {step}{Step}

\newcommand{\ccn}{\mathcal{C}_n(R)}

\newcommand{\hk}{(H_k,\lambda_k)_{1 \leq k \leq b}}

\newcommand{\G}{\mathbb{G}}
\newcommand{\R}{\mathbb{R}}

\newcommand{\N}{\mathbb{N}}
\newcommand{\cc}{\mathscr{C}}\newcommand{\C}{\mathcal{C}}

\newcommand{\s}{\mathcal{S}}
\newcommand{\tim}{{t \in \R^m}}

\newcommand{\smn}{\mathcal{S}_{m+n}}

\newcommand{\xt}{{\tilde{x}}}

\newcommand{\D}{\mathcal{D}}

\newcommand{\ep}{\varepsilon}

\newcommand{\pa}{\partial}

\newcommand{\Pa}{\mathcal{P}}

\newcommand{\bou}{\mathbf{B}}
\newcommand{\sph}{\mathbf{S}}

\newcommand{\supp}{\mbox{\rm supp}}

\title[]{Regular vectors and bi-Lipschitz trivial stratifications in o-minimal structures}

\makeatletter
\@namedef{subjclassname@2020}{%
\textup{2020} Mathematics Subject Classification}

\@addtoreset{equation}{section}

\makeatother

\author[ G. Valette]{ Guillaume Valette}

\address[G. Valette]{Instytut Matematyki Uniwersytetu
Jagiello\'nskiego, ul. S. \L ojasiewicza 6, Krak\'ow, Poland}\email{guillaume.valette@im.uj.edu.pl}

 \thanks{Research partially supported
by the NCN grant  2021/43/B/ST1/02359.}


\begin{document}

\begin{abstract} These notes focus on the Lipschitz geometry of sets that are definable in o-minimal structures (expanding the real field).   We show that every set which is definable in a polynomially bounded o-minimal structure admits a stratification which is locally definably bi-Lipschitz trivial along the strata. This result is obtained as a byproduct of two foregoing results of the author. The first one asserts that, given a family definable in an o-minimal structure, there is a regular vector, up to a definable family of bi-Lipschitz homeomorphisms. The second one is a bi-Lipschitz version of the famous Hardt's theorem. We give proofs of these two theorems that avoid the use of the real spectrum.  The article recalls the basic facts and the results about Lipschitz geometry that are needed to understand the proofs, providing references.
\end{abstract}
\maketitle
\begin{section}{Introduction}
The study of the Lipschitz geometry of singularities that arise in algebraic and analytic geometry began when T. Mostowski constructed stratifications of complex analytic sets that admit a Lipschitz version of Thom-Mather isotopy theorem  \cite{most}. This result was extended by A. Parusi\'nski to real analytic geometry  \cite{parureal,parusub}, and his proof was then adapted to polynomially bounded o-minimal structures \cite{lipsomin}. The Lipschitz geometry of singularities was investigated later independently by the author of the present paper \cite{vlt,link,sullivan, linfty,gvpoincare} (see \cite{livre} for a complete expository), as well as by several other authors \cite{bfg, bfgo, bgg, bnp,kdec, kp, keq} (among many others).

In \cite{vlt}, the author proved a bi-Lipschitz triviality theorem, which can be considered as a bi-Lipschitz  version of Hardt's theorem (recalled in Theorem \ref{hardt} below). These  notes provide a new proof of this theorem and then derive existence of definably locally bi-Lipschitz trivial stratifications.
We also include a short introduction to the theory of the Lipschitz geometry of sets that are definable in o-minimal structures,   providing many necessary definitions, proofs, and references.


The existence of locally bi-Lipschitz trivial stratifications and the description of the aspect of singularities occurring in tame geometry (subanalytic, semialgebraic, or o-minimal) from the metric point of view that was achieved during the four last decades  recently turned out to be valuable for applications to analysis of PDE's and geometric measure theory.
In \cite{linfty,gvpoincare}, the author relied on it so as to compute the $L^p$ cohomology of differential forms of bounded subanalytic manifolds, not necessarily compact.  More recently, these techniques turned out to be useful to study the theory of currents \cite{depauw}, as well as to investigate the Sobolev spaces of these manifolds \cite{ guillermou, poincwirt,trace, lprime}, which is valuable for applications to  the theory of PDE on domains with non Lipschitz boundary \cite{laplac}.


 The main difficulty of the proof of the bi-Lipschitz version of Hardt's theorem \cite{vlt} (see Theorem \ref{hardt} below) is the ``regular vector theorem'' \cite[Theorem $3.13$]{vlt} (see Theorem \ref{thm_proj_reg_hom_pres_familles} and Corollary \ref{cor_proj_reg_hom_pres}). This theorem asserts that, given a set $A\subset \R^n$ which is definable in an o-minimal structure, there is a definable bi-Lipschitz homeomorphism $h:\R^n\to \R^n$ such that $h(A)$ has a regular vector, which means that there is a vector such that the angle between it and all the tangent spaces to $h(A)$ (at its regular points) is bounded away from zero. We prove this result, which does not require the structure to be polynomially bounded and which is of its own interest. The regular vector theorem was actually also the key ingredient of the Lipschitz conic structure theorem in \cite{gvpoincare}, itself useful later to study Sobolev spaces of definable manifolds \cite{poincwirt,trace, lprime}, and was of service to M. Czapla  to show existence of triangulations inducing Whitney stratifications \cite{wt}.



 Quite often, especially if, like in the aforementioned applications of Lipschitz geometry,  Lipschitz isotopies are necessary, one needs a regular vector not only for one single set but for a definable family of sets. In \cite{vlt}, this is obtained by applying the regular vector theorem to the generic fiber of a family,  relying on the compactness of the Stone space of the Boolean algebra of definable sets (sometimes called the real spectrum). This has nevertheless the inconvenience to involve abstract material to which specialists of PDE or geometric measure theory may be unfamiliar, and to force to work with an o-minimal structure that expands an arbitrary real closed field, that may be non archimedean and totally disconnected, which is prone to generate technical complications.

 In the present article, we give a parameterized version of the regular vector theorem (Theorem \ref{thm_proj_reg_hom_pres_familles}) on o-minimal structures following the proof given in \cite{vlt}, but relying only on very elementary methods. We also combine it with the techniques of \cite{gvpoincare} to prove a local version, with additional properties (Theorem \ref{thm_proj_loc}), which was used in the latter article to prove the Lipschitz conic structure of subanalytic sets. We then provide a proof  of the bi-Lipschitz version of Hardt's Theorem on polynomially bounded o-minimal structures (Theorem \ref{hardt}), and derive existence of definably bi-Lipschitz trivial stratifications (Corollary \ref{cor_strat_bil_triv}).

 \noindent {\bf Some notations and definitions.}
 Throughout this article,  $m,n$, $j$, and $k$ will stand for  integers.
The origin of $\R^n$ will be denoted $0_{\R^n}$. When the ambient space will be obvious from the context, we will however omit the subscript $\R^n$. We write $e_1,\dots,e_n$ for the canonical basis of $\R^n$.

We write $|x|$  for the euclidean norm and $d(x,y)$ for the euclidean distance (and the distance to a subset $P\subset \R^n$ will be denoted by $d(x,P)$).   Given $x\in  \R^n$ and $\ep>0$, we denote by $\bou(x,\ep)$ the open ball of radius $\ep$  centered  at $x$ (for the euclidean norm). The unit sphere of $\R^n$ centered at the origin is denoted $\sph^{n-1}$.  Given a subset $A$ of $\R^n$, we respectively denote the closure and interior of $A$ by $cl(A)$ and $int(A)$, and set $\delta A=cl(A)\setminus int (A)$.

A mapping $\xi:A \to \R^k$ is said to be {\bf Lipschitz}\index{Lipschitz} if there is a constant $L$ such that for all $x$ and $x'$ in $A$: $$|\xi(x)-\xi(x')|\le L|x-x'|.$$ We say that $\xi$ is {\bf $L$-Lipschitz}\index{l-Lipschitz@$L$-Lipschitz} if we wish to specify the constant. 
The smallest nonnegative number $L$ having this property is called the {\bf Lipschitz constant of $\xi$} and is denoted $L_\xi$.
 By convention, if $A$ is empty then $\xi$ is Lipschitz and $L_\xi=0$. A mapping $\xi$ is {\bf bi-Lipschitz} if it is a homeomorphism onto its image such that $\xi$ and $\xi^{-1}$ are both Lipschitz.

 Given two functions $\zeta$ and $\xi$ on a set $A \subset
\R^n$ with $\xi \leq \zeta$ we define the {\bf closed band $[\xi,\zeta]$} as the set:
$$[\xi,\zeta]:=\{(x,y)\in A \times \R: \xi(x)\leq y \leq \zeta(x)\}.$$
The open and semi-open bands $(\xi,\zeta)$, $(\xi,\zeta]$, and $[\xi,\zeta)$, are then defined analogously. 
 
  Given a subset $B$ of $A$, we write ``$\xi\lesssim \zeta$ on $B$'' when there is a constant $C$ such that $\xi(x) \le C\zeta(x)$ for all $x\in B$. We write ``$\xi\sim \zeta$ on $B$'' or ``$\xi(x)\sim \zeta(x)$ for $x$ in $B$'' whenever both $\xi\lesssim \zeta$ and $\zeta \lesssim \xi$ hold on $B$.
\end{section}
\begin{section}{O-minimal structures}
A \textbf{structure} (expanding the field $(\mathbb{R}, + ,.)$) is a family $\s = (\s_n)_{n \in \mathbb{N}}$ such that for each $n$ the following properties hold
\begin{enumerate}
\item [(1)] $\s_n$ is a boolean algebra of subsets of $\mathbb{R}^n$,
\item[(2)] If $A \in \s_n$ then $\mathbb{R}\times A$ and $A \times \mathbb{R}$ belong to $\s_{n+1}$,
\item[(3)] $\s_n$ contains $\{x \in \mathbb{R}^n: P(x) = 0\}$, for all $P \in \mathbb{R}[X_1,\ldots, X_n]$,
\item[(4)] If $A \in \s_n$ then $\pi(A)$ belongs to $\s_{n-1}$, where $\pi: \mathbb{R}^n \to \mathbb{R}^{n-1}$ is the standard projection onto the first $(n-1)$ coordinates.
 \end{enumerate}
 A structure $\s$ is said to be \textbf{o-minimal} if in addition:
 \begin{enumerate}
 \item [(5)] Any set $A\in \s_1$ is a finite union of intervals and points.
 \end{enumerate}
A set belonging to  $\s_n$, for some $n$, is called a {\bf definable set}, and a map whose graph is in some $\s_n$ is called a {\bf definable map}.

A structure $\s$ is said to be \textbf{polynomially bounded} if for each definable function $f: \mathbb{R}\to \mathbb{R}$, there exists a positive number $a$ and an $n \in \mathbb{N}$ such that $|f(x)| < x^n$ for all $x > a$.

{\it We fix an o-minimal structure $\s$ for all this article.} It will be assumed to be polynomially bounded in section \ref{sect_triviality}. For the other sections, this assumption is unnecessary.

We refer to \cite{costeomin, vdd} for all the basic facts and definitions about o-minimal structures that we shall use all along this article, such as cell decompositions or curve selection lemma. We however recall a few definitions helpful to understand the statements of the theorems.

It is a fundamental feature of o-minimal structures that it is possible to construct a cell decomposition of $\R^n$ that is {\bf compatible} with a given arbitrary finite collection of elements of $\s_n$, in the sense that the given sets are unions of cells of this decomposition (the word ``adapted'' is used in \cite{costeomin}, instead of compatible).

 The definition of cell decompositions being inductive on the dimension of the ambient space, it is obvious that if $\C$ is a cell decomposition of $\R^n$ and if $\pi:\R^n \to \R^k$ (with $k\le n$) is the canonical projection, then $\{\pi(C):C\in \C\}$ is a cell decomposition. We will denote it by $\pi(\C)$. 

 A {\bf 
stratification} of a subset of $ \R^n$ is a finite partition of it into
definable smooth submanifolds of $\R^n$, called {\bf strata}. A {\bf stratification is compatible} with a set if this set is the union of some strata. It is well-known that we can construct Whitney $(b)$ or Verdier $(w)$ regular stratifications compatible with any given definable set \cite{taleloi, livre}.


\begin{dfn}\label{dfn_familles}
 We say that $(A_t)_{t\in\R^m}$ is a {\bf definable family of subsets of} $\R^n$ if the set 
$$A:=\bigcup_{t\in\R^m}\{t\}\times A_t$$ is a definable subset of $\R^m\times\R^n$. 

We will sometimes regard a definable subset $A\subset\R^m\times\R^n$ as a definable family,
setting for $\tim$: $$A_t:=\{x\in\R^n\,: (t,x)\in A\}.$$ 

Given two definable families $A\subset\R^m\times\R^n$ and $B\subset\R^m\times\R^k$, we say that $F_t:A_t \to B_t$, $t\in \R^m$, is a {\bf definable family of mappings} if the family of the graphs $(\Gamma_{F_t})_{t\in \R^m}$, is a definable family of subsets of $\R^{n+k}$.  We will sometimes regard a function $f:A\to \R$, $A\in \s_{m+n}$, as a family of functions $f_t:A_t\to \R$, $\tim$, setting $f_t(x):=f(t,x)$. 

A definable family of mappings $F_t:A_t\to B_t$, $\tim$ is {\bf uniformly Lipschitz (resp. bi-Lipschitz)} if there exists a constant $L$ such that
 $F_t$ is $L$-Lipschitz  (resp. $F_t$ and $F_t^{-1}$ are $L$-Lipschitz) for all $\tim$. 
\end{dfn}

Given $B\in\s_m$ and $A$ as above, we also define the {\bf restriction of $A$ to $B$}:
\begin{equation}\label{eq_rest_familles}A_B:=A \cap  (B\times \R^n).\end{equation}
                              
Define finally the {\bf $m$-support} of  $A$ by $$\supp_m(A):=\{t \in \R^m :A_t \ne \emptyset\}.$$

\begin{dfn}\label{dfn_trivialite}
Let $A\in \s_{m+n}$. We will say
that $A$ is \textbf{definably topologically trivial along} $U
\subset \R^m$ if there exist $t_0 \in U$ and a definable  
homeomorphism $H :U \times  A_{t_0}  \rightarrow A_U$, $(t,x)\mapsto (t,h_t(x))$. The mapping $h$ is then called {\bf the trivialization} of the set $A$ along $U$.
\end{dfn}

The following theorem is sometimes called ``definable Hardt's theorem'', because it is the  o-minimal counterpart of a theorem proved by R. Hardt about semialgebraic families of sets \cite{hardt}. In this theorem, by {\bf definable partition of a set}, we mean a {\it finite} partition of it into definable sets.

\begin{thm}\label{thm_hardt_C0}\cite[Theorem 5.22]{costeomin}
Given $A\in \s_{m+n}$, there
exists a definable partition of $\R^m$ such that  $A$
is definably topologically trivial along each  element of this
partition.
\end{thm}

\begin{rem}\label{rem_cc}
We shall make use of the following immediate consequence of this theorem: given $A\in \s_{m+n}$, there is a definable partition $\Pa$ of $\R^m$ such that, for every $B \in \Pa$,  $E_t$  is connected for every connected component $E$ of $A_B$ and all $t \in B$. 
\end{rem}

\end{section}


\begin{section}{The regular vector theorem}\label{sect_main_res}
 We denote by $\G^{n}_k$ the Grassmannian of $k$-dimensional vector subspaces of $\R^n$, and we set $\G^n:=\bigcup_{k=1}^{n}\G^n_k$ as well as $\G^n_*:=\bigcup_{k=1}^{n-1}\G^n_k$.  
 
 Given a
definable set $A \subset \R^n$, we denote by $A_{reg}$ the set constituted by all the points of $A$ at which this set is a $\cc^1$ manifold (without boundary, of dimension $\dim A$ or smaller).
Define $\tau (A)$  as the
closure of the set of vector spaces that are tangent to $A$ at a
regular point, i.e.:
$$\tau(A):=cl ( \{ T_x A  \in \G^n : x \in A_{reg} \}).$$

 Given an element $\lambda$ of $\sph^{n-1}$ and a subset $Z \subset \G^n$  we set (caution, here $Z$ is not a subset of $\R^n$):$$d(\lambda,Z) := \inf \{d(\lambda,T):T \in Z\},$$ with $d(\lambda, \emptyset):=+\infty$.
\begin{dfn}\label{boule_reguliere}
Let $A\in \s_n$. An element $\lambda$ of
$\sph^{n-1} $ is said to be {\bf regular  for the set $A$}\index{regular! for a set}  if there is $\alpha >0$ such that:
$$
d(\lambda, \tau(A)) \geq \alpha.
$$
More generally, we say that $\lambda \in \sph^{n-1} $ is {\bf regular for  $A\in \s_{m+n}$}  if there exists $\alpha >0$ such that for any  $t\in \R^m$:
\begin{equation}\label{eq_reg_familles}
  d(\lambda, \tau(A_t)) \geq \alpha.
 \end{equation}
We then also say that $\lambda$ is {\bf regular for the family $(A_t)_{t \in \R^m}$}.
A {\bf subset $C \subset \sph^{n-1}$ is regular}  for a set  $A\in \s_{m+n}$ if so are all the elements of $cl(C)$. 
\end{dfn}

 If $\lambda\in \sph^{n-1}$ is regular for $A\in \smn$, it is regular for $A_t\in \s_n$ for all $t \in \R^m$. But it is indeed even stronger since 
in (\ref{eq_reg_familles}), the  angle between the vector $\lambda$ and the tangent spaces to the fibers is required to be bounded below away from zero by a positive constant {\it independent of the parameter $t$}.

Regular vectors do not always exist, even if the considered set has empty interior (which is clearly a necessary condition), as it is shown by the simple example of a circle. Nevertheless, when the considered sets have empty interior, up to a definable bi-Lipschitz map, we can find such a vector:
\begin{thm}\label{thm_proj_reg_hom_pres_familles}
Let $A\in \smn$ such that $A_t$ has empty interior for every $t \in \R^m$. There exists a uniformly bi-Lipschitz 
  definable family of homeomorphisms  $h_t :  \R^n \rightarrow  \R^n$, $\tim$,  such that  $e_n$   is regular for the family $(h_t(A_t))_{t\in \R^m}$.
\end{thm}

\begin{rem}\label{rem_proj_reg_parametres}
In the above theorem, the family $h_t$ is not required to be Lipschitz nor continuous with respect to the parameter $t \in \R^m$. It is nevertheless continuous for generic parameters (see \cite[Lemma 5.17]{costeomin} or \cite[Proposition 2.4.9]{livre}). 
Also, using Proposition \ref{pro_lipschitz_parametres} below, one could see that, along the elements of a suitable partition  of $\R^m$, $h$ and $h^{-1}$ may be required to be Lipschitz with respect to the parameters  on compact sets.  
\end{rem}


In the case $m=0$, we have the following immediate corollary which was proved in \cite{vlt}:

\begin{cor}\label{cor_proj_reg_hom_pres}
Let $A\in \s_n$ be of empty interior. There exists a
definable  bi-Lipschitz homeomorphism $h : \R^n \rightarrow \R^n$
such that $e_n$ is regular for $h(A)$.
\end{cor}

The example of a circle that we already mentioned points out the fact that it is not possible for the homeomorphism given by this corollary to be always smooth, even if so is $A$.

  The proof of this theorem is given in section \ref{sect_proof_main}.  In order to motivate the material that we are going to introduce for this purpose in sections \ref{sect_a_few_lemmas} and \ref{sect_regular_systems} (especially Definition \ref{dfn_systemes} and Theorem \ref{thm_existence_des_familles}), let us now give a brief outline of the construction of this homeomorphism (we assume $m=0$ in the outline for simplicity), with explicit references to the key results and definitions.

It is actually easy to see that given $A\in \s_n$, there is a covering of $\R^n$ by finitely many sets, say $G_1,\dots, G_l$, such that each $G_k\cap A$ has a regular vector $\lambda_k$.  As a matter of fact, for each $k$,  there is a linear automorphism $h_k$ such that the vector $e_n$ is regular for $h_k(G_k\cap A)$. The problem is that it is not easy to ``paste'' these ``local embeddings'' $h_1,\dots,h_l$ into a bi-Lipschitz map $h:\R^n\to \R^n$.   Somehow, the idea will be to define $h$ on $\bigcup_{i=1}^k G_i$ inductively on $k$, by means of the $h_i$'s, starting with $h=h_1$.

We introduce for this purpose an ``induction machinery'', called {\it regular systems of hypersurfaces} (Definition \ref{dfn_systemes}).  Extending $h$ from  $\bigcup_{i=1}^k G_i$ to  $\bigcup_{i=1}^{k+1} G_i$ somehow requires to change coordinates, and the transition map $h_{k+1}\circ h_k^{-1}$  can be interpreted as a turn from the direction  $\lambda_k$ to the direction $\lambda_{k+1}$.
 These turns could make it difficult to extend a bi-Lipschitz  mapping to a bi-Lipschitz mapping for we might come back to our starting  point. Working with a regular  system of hypersurfaces $H$ makes it possible to turn ``without turning back'' (see (\ref{eq_E_lambda}) as well as (\ref{item_reg_syst_ii}) of Definition \ref{dfn_systemes}), progressing in a zigzag but somehow always going toward the same Lipschitz upper half-space $G_b(H)$. 

The main difficulty is therefore the proof of Theorem \ref{thm_existence_des_familles}, which yields existence of a suitable regular system of hypersurfaces. In the proof of this theorem, the trick to avoid to ``turn back'' (in the sense of  (\ref{item_reg_syst_ii}) of Definition \ref{dfn_systemes}) is to choose $\lambda_{k+1}$ in the same connected component as $\lambda_k$ of the sets of all the regular vectors of the previous step (see Proposition \ref{pro_connexite_des_projections}). The key lemma on this issue is Lemma \ref{lem_direction_generique_droit_proj}, which relies on the fact that the fiber $\widetilde{\pi}_e^{-1}(\lambda)$, for each $e$ and $\lambda$ in $\sph^{n-1}$ (see (\ref{eq_pitilda})), is a connected curve of length at least $2$, which leaves enough space to choose our regular vector (see Remark \ref{rem_connexite_pi_tilda}).

The proof of Theorem  \ref{thm_existence_des_familles} is splitted into four steps, and a  more explicit outline of it  is provided before the first step (see section \ref{sect_proof}). Moreover, the second and third steps are preceded by a paragraph that motivates them and gives some more details on their proof.
\end{section}

\begin{rem} Mostowski's work \cite{most} involved establishing results about existence of regular projections  (see also \cite{kdec,oudrane, pdec, parusub}).
Existence of a regular vector is closely related to existence of a regular projection but not completely equivalent \cite{nhan, oudrane}. This is however not the main difference between Corollary \ref{cor_proj_reg_hom_pres} and the theorems of \cite{most,oudrane, pdec, parusub}:
 Corollary \ref{cor_proj_reg_hom_pres}  provides one single vector which is regular for the whole image of the considered set by some definable bi-Lipschitz mapping whereas the theorems of \cite{most,  oudrane, pdec, parusub} provide a finite set of projections such that, at each point of the considered set itself, at least one of them is regular.
\end{rem}


\begin{section}{A few lemmas on Lipschitz geometry}\label{sect_a_few_lemmas}
\begin{subsection}{Regular vectors and Lipschitz functions.}\label{sect_reg_lips}
  Proving Theorem \ref{thm_proj_reg_hom_pres_familles} will require to prove parameterized versions of all the lemmas and propositions of \cite{vlt}. In \cite{kp},  K. Kurdyka and A. Parusi\'nski provided a parameterized version of the ``$L$-regular cell decomposition theorem''  \cite{kdec}, which enabled them to generalize their proof of Thom's gradient conjecture on o-minimal structures.  We start with a result of \cite{kp} that  will be useful for our purpose.
  
  Given $a$ and  $b$ in a definable connected set $A$, let:$$d_{A}(a,b):=\inf \{ length(\gamma):\gamma:[0,1]\to A, \mbox{ $\cc^0$ definable arc joining } a \mbox{ and } b\}$$
  (as definable arcs are piecewise $\cc^1$, their length is well-defined).  This defines a metric on $A$, generally referred as {\bf the inner metric} of $A$. 
  
   To avoid any confusion, we will refer to the restriction to $A$ of the euclidean
 metric as the {\bf outer metric}.  When $A$ is smooth, a $\cc^1$ function that has bounded derivative is Lipschitz with respect to the inner metric, but is not necessary  Lipschitz (w.r.t. the outer metric),  these two metrics being not always equivalent. We however have the following result \cite[Theorem 1.2]{kp}:

    \begin{thm}Every $A\in \s_{m+n}$ admits a definable partition into cells, such that for each $E\in \Pa$ and each $\tim$, the inner and outer metrics of $E_t$ are equivalent. The constants of this equivalence just depend on $n$ (and not on $m$ and $t$). 
\end{thm}\label{thm_kp}
 The techniques that we use in section \ref{sect_finding} to prove  Proposition \ref{pro_dec_L_regulier_avec_proj_fixees}  are actually related to the main ideas of the proof of this theorem that is given in \cite{kp}. It is indeed possible to show the above theorem from the latter proposition, together with an induction on $n$. For more details we refer the reader to the latter article (see also \cite[Chapter 3]{livre}).

\begin{pro}\label{pro_extension_fonction_lipschitz} Every definable Lipschitz function $\xi:A \to \R$,  $A \in \s_n$, can be extended to an $L_\xi$-Lipschitz definable function  $\tilde{\xi}:\R^n \to \R$.
 \end{pro}
\begin{proof}
Set $\widetilde{\xi}(q):=\inf \{\xi(p)+L_\xi |q-p|:p\in A\}$. By the quantifier elimination principle, it is a definable function. An easy computation shows that it is $L_\xi$-Lipschitz.
\end{proof}

\begin{rem}\label{rem_extension_familles_fonctions_lipschitz}
Let $A \in \s_{m+n}$  and let a definable function $\xi:A\to  \R$ be such that $\xi_t:A_t \to \R$ is a Lipschitz function for every $t\in \R^m$. 
The respective extensions $\tilde{\xi}_t$  of $\xi_t$, $t \in \R^m$ (with for instance $\tilde{\xi}_t\equiv 0$ if $t \notin \supp_m A$), provided by the proof of the  above proposition constitute a definable family of functions. We thus can extend definable families of Lipschitz functions to definable  families of Lipschitz functions. This will be of service.
\end{rem} 
 
\begin{lem}\label{lem_reg_inclusion}
 Let $A$ and $B$ in $\s_{n+m}$ with $B\subset A$.  If $\lambda \in \sph^{n-1}$ is regular for $A$, then it is regular for $B$.
 \end{lem}
\begin{proof}
 Assume that $\lambda \in \sph^{n-1}$ is not regular for $B$. It means that there is a sequence $((t_i,b_i))_{i \in \N}$, with $b_i \in B_{t_i,reg}$ such that $\tau:= \lim T_{b_i}B_{t_i,reg}$ exists and contains $\lambda$. Choose for every $i$ a Whitney $(a)$ regular stratification of $A_{t_i}$ (see for instance \cite{bcr,livre} for the definition) compatible with $B_{t_i}$ and $B_{t_i,reg}$  and denote by $S_i$ the stratum containing $b_i$.  Moving slightly $b_i$ if necessary, we may assume that $S_i$ is open in $B_{t_i,reg}$ (since $B_{t_i,reg}$ is open and dense in $B_{t_i}$), which entails that  $T_{b_i} S_i=T_{b_i}B_{t_i,reg}$.  As $A_{t_i,reg}$ is dense in $A_{t_i}$, for every $i\in \N$, we can find $a_i$ in $A_{t_i,reg}$, which is close to $b_i$.   Moreover, possibly extracting a sequence, we may assume that $\tau':= \lim T_{a_i}A_{t_i,reg}$ exists. 
 If $a_i$ is sufficiently close to $b_i$, by Whitney $(a)$ condition, we deduce that $\tau' \supset \tau$, which contains $\lambda$. This yields that $\lambda$ is not regular for $A$.
\end{proof}
\begin{rem} It is worthy of notice that the proof of the above lemma has established that the corresponding number $\alpha$ (see (\ref{eq_reg_familles})) can remain the same for $B$. 
\end{rem}

Given  $\lambda \in \sph^{n-1} $, we denote by $\pi_\lambda :\R^n
\rightarrow N_\lambda$  the orthogonal  projection onto the
hyperplane $N_\lambda$ normal to the vector $\lambda$, and by $q_\lambda$ the coordinate
of $q\in \R^n$ along $\lambda$, i.e. the number given by the euclidean inner product of $q$ with $\lambda$.

Given $B\in \s _{n}$ and  $\lambda \in \sph^{n-1} $, with $B \subset
 N_\lambda$, as well as
 a function $\xi:B \to \R$, we set 
\begin{equation}\label{eq_graph_for_lambda}
\Gamma_\xi ^\lambda:=\{ q \in \R^n : \pi_\lambda(q)  \in B
\quad  \mbox{and} \quad q_\lambda = \xi(\pi_\lambda(q))\},
\end{equation} and call this set {\bf the graph of $\xi$ for $\lambda$}. 

\begin{pro}\label{pro_proj_reg_decomposition_en_graphes}
The vector $\lambda \in \sph^{n-1}$ is regular for the set $A \in \smn$  if and only if there are finitely many uniformly Lipschitz definable families of functions $\xi_{i,t}: B_{i,t} \to  \R$, $\tim$, with $B_i \subset \R^m \times N_\lambda$, $i=1,\dots,p$,  such that for all $\tim$:
$$A_t = \bigcup_{i=1}^p\Gamma_{\xi_{i,t}}^\lambda.$$
\end{pro}
\begin{proof} As the ``if'' part is clear, we will focus on the converse.
 Up to a linear isometry we can assume that $\lambda=e_n$. Let  $A \in \smn$. Take a cell decomposition compatible with $A$ and let $C$ be a cell included in $A$. This cell cannot be a band since $e_n$ is regular for $A$ (see Lemma \ref{lem_reg_inclusion}). It is thus the graph of a $\cc^1$ function $\xi:D \to \R$, with $D \in \s_{m+n-1}$, such that $\xi_t$ has bounded first derivative (independently of $t$). It therefore must be uniformly Lipschitz with respect to the inner metric. It follows from Theorem \ref{thm_kp} that there is a definable partition $\Pa$ of $D$ such that for each $E\in \Pa$, the inner metric of $E_t$ is equivalent to its outer metric for all $\tim$, with constants that just depend on $n$. The family of functions $\xi$ induces a uniformly Lipschitz family of functions on every element of $\Pa$. 
\end{proof}

We finish this subsection with an elementary proposition that will be of service to prove Theorem \ref{thm_existence_des_familles}. This proposition yields that we can replace a given collection of families of Lipschitz functions with an increasing collection of families of Lipschitz functions $\xi_{1,t}\le \dots \le \xi_{k,t}$ in such a way that the union of the graphs is unchanged:

\begin{pro}\label{pro_famille_lipschitz_ordonnees}
Let $f_{1,t},\dots,f_{k,t}$, $\tim$, be definable families of functions on $ N_\lambda$, $\lambda \in \sph^{n-1}$, and let $L \in \R$.  Assume that for all $i\le k$ and for all $t \in \R^m$, the function $f_{i,t}$ is  $L$-Lipschitz.   Then, there exist some definable families of functions $\xi_{1,t},\dots,\xi_{k,t}$ on $ N_\lambda$
 such that for all $t \in \R^m$
 \begin{enumerate}
\item $\xi_{i,t}$ is $L$-Lipschitz and all $i \le k$.
\item  $\bigcup_{i=1}^k \Gamma_{\xi_{i,t}}^\lambda=\bigcup_{i=1}^k \Gamma_{f_{i,t}}^\lambda$.
\item $\xi_{1,t}\le \dots \le \xi_{k,t}$.
\end{enumerate}
\end{pro}
\begin{proof}
Up to an orthogonal linear mapping with may assume $\lambda=e_n$.  We are going to define  inductively on $j$ some definable  integer valued functions $i_j:\R^m \times \R^{n-1} \to \R$, $j=1,\dots,k$ such that for every $t \in \R^m$ and $j \le k$, the functions   \begin{equation}\label{eq_g_j}\xi_{j,t}(x):=f_{i_j(t,x),t}(x)\end{equation} are $L$-Lipschitz functions   satisfying  $\xi_{1,t}\le \dots \le \xi_{k,t}$. Indeed, let $i_1(t,x):=\min \{ i \le k: f_{i,t}(x) =\min_{l\le k} f_{l,t}(x) \}$. Then, assuming that $i_1,\dots,i_{j-1}$ have been defined, let $$i_j(t,x):= \min \{ i \in I_j(t,x): f_{i,t}(x) =\min_{l \in I_j(t,x)} f_{l,t}(x) \},$$ where $I_j(t,x)$ is the set constituted by the positive integers which are not greater than $k$ and different from $i_1(t,x),\dots,i_{j-1}(t,x)$. We clearly have $\xi_{1,t}(x)\leq \dots \le \xi_{k,t}(x)$ if $\xi_{j,t}(x)$ is defined as in (\ref{eq_g_j}).

  Take a cell decomposition $\C$ of $\R^m \times \R^{n-1}$ such that the functions $(f_{j,t}(x)-f_{j',t}(x))$ have constant sign (positive, negative, or zero)  on every cell and observe that, since the $i_j$'s are constant on every cell, they are definable.

By construction, we have $\bigcup_{j=1} ^k \Gamma_{\xi_{j,t}} =\bigcup_{j=1}^k \Gamma_{f_{j,t}}$, for all $\tim$, which entails that $e_n$ is regular for the graphs of the families $\xi_{j,t}$, $\tim$. As a matter of fact, for each $j$, in order to show that $\xi_{j,t}$ is  $L$-Lipschitz, it suffices to establish that the functions $\xi_{j|C_t}$, $C \in \C$, glue together into a continuous function on $\R^{n-1}$ for every $t$, which is left to the reader. 
\end{proof}

\end{subsection}

\begin{subsection}{Finding regular directions.}\label{sect_finding}

\begin{lem}\label{lem_kurdyka}
Given $\nu \in \mathbb{N}$, there exists   $t_\nu >0$ such that for any
$P_1,\dots, P_\nu$ in $\G^n_*$ there exists a vector $\lambda \in \sph^{n-1} $ such that for any  $i$:
$$d(\lambda,P_i) > t_\nu.$$
\end{lem}
\begin{proof}
Given $P_1,\dots, P_\nu$ in $\G^n_*$, let $\varphi(P_1,\dots,P_\nu):=\sup_{\lambda\in \sph^{n-1}}\min_{i\le \nu} d(\lambda,P_i)$. Since the $P_i$'s have positive codimension, $\varphi$ is a positive function, which, since the Grassmannian is compact, must be bounded below away from zero. 
\end{proof}

The next lemma is a refinement of the just above lemma which says that the vector $\lambda$ can be chosen among finitely many ones.

\begin{lem}\label{lem  kurdyka proj fixees}
 Given $\nu \in \N$, there exist $\lambda_1, \dots, \lambda_N$ in $\sph^{n-1}$
and $\alpha_\nu >0$ such that for any $P_1 , \dots, P_\nu$ in $\G^n_*$ we may find $i \leq N$ such that for any $j \le \nu$:
$$ d(\lambda_i  , P_j)> \alpha_\nu.$$
\end{lem}
\begin{proof}
Let $t_{\nu}$ be the real number given by Lemma \ref{lem_kurdyka} and
let $\lambda_1, \dots, \lambda_N$ in $\sph^{n-1}$ be such that
 $\bigcup_{i=1}^N \bou(\lambda_i,\frac{t_\nu}{2})\supset \sph^{n-1}$.  Suppose that there are  $P_1,
 \dots, P_\nu$ in $\G^n_*$ such that for any $i \in \{1,\dots N\}$
 we have $d(\lambda_i
 ,\bigcup_{j=1} ^\nu P_j)\le \frac{t_\nu}{2}$. This
 implies that  
 any $\lambda$ in $\sph^{n-1}$ satisfies  $$d(\lambda
 ,\bigcup_{j=1} ^\nu P_j)<  t_\nu,$$ contradicting Lemma \ref{lem_kurdyka}. It is thus enough to set $\alpha_\nu:=\frac{t_\nu}{2}$.
\end{proof}
The next lemma will require a definition.
We  estimate the   angle between two vector subspaces $P$ and $Q$ of $\R^n$ in the following
way:
$$\angle (P,Q)=\sup \{ d(\lambda,  Q): \lambda\: \, \mbox{is a unit vector of}\:\, P\}.$$
 This constitutes a metric on each $\G^{n}_k$,  $k\le n$.

\begin{dfn}\label{df alpha flat}
Let $\alpha >0$ and $Z\in \s_{m+n}$. We say that the family $(Z_t)_{\tim}$ is
{\bf $\alpha$-flat}\index{flat@$\alpha$-flat} if: $$\sup \{\angle (P,Q):P, Q \in \bigcup_\tim \tau(Z_{t,reg})\} \le \alpha.$$
We then also say that $Z$ is {\bf $(m,\alpha)$-flat}\index{flat@$(m,\alpha)$-flat}. When $m=0$, we say that $Z$ is {\bf $\alpha$-flat}. 
\end{dfn}

If $P$ and $Q$ are two vector subspaces of $\mathbb R^n$ satisfying $\dim P>\dim Q$ then $\angle (P,Q)=1$. As a matter of fact, if $Z$ is $(m,\alpha)$-flat for some $\alpha<1$, then $Z_t$ must be of pure dimension for all $t$.

\begin{rem}\label{rem_alpha_flat}
 It follows from Lemma \ref{lem  kurdyka proj fixees} that if $Z_{1,t},\dots, Z_{\nu,t}$, $\tim$, are $\alpha_{\nu}$-flat definable families (where $\alpha_\nu$ is the constant provided by the latter lemma) of subsets of $\R^n$ of empty interiors then one of the $\lambda_i$'s (that are also provided by the latter lemma) is regular for all these families. 
\end{rem}

\begin{lem}\label{lem_partition_tangents}
Given $Z\in \s_{m+n}$ and $\alpha>0$, we can find a finite partition  of $Z$ into $(m,\alpha)$-flat sets. 
\end{lem}
\begin{proof}Dividing $Z$ into cells, we may assume that $Z_t$ is a manifold for all $\tim$.
We can cover the Grassmannian by finitely many balls of radius $\frac{\alpha}{2}$, which gives rise to a covering $U_1,\dots,U_k$ of $Z$  (via  the family of  mappings $Z_{t}\ni x \mapsto T_x Z_{t}$) by  $(m,\alpha)$-flat sets. 
\end{proof}

This leads us to the following result that originates in \cite{vlt} and that will be of service in section \ref{sect_triviality}. It is closely related to the $L$-regular cell decompositions introduced and constructed in \cite{kdec}. The difference is that we wish that the regular vector for $\delta C$ can be chosen among finitely many ones. This result was then improved by W. Paw\l ucki \cite{pawlucki}  who has shown that we can require in addition $N=n$.

\begin{pro}\label{pro_dec_L_regulier_avec_proj_fixees}
There exist $ \lambda_1, \dots, \lambda_N$ in  $\sph^{n-1}$
such that for any $A_1,\dots, A_p$   in  $\smn$,
there is a cell decomposition $\C$ of
$\R^{m+n}$ compatible with all the  $A_k$'s and such
that for each cell $C \in \C$ satisfying $\dim C_t=n$ (for all $t\in \supp_m C$), we may find $\lambda _{j(C)}$, $ 1
\leq j(C) \leq N$,
regular for the family $ (\delta C_t)_\tim$.
\end{pro}
\begin{proof}
According to Lemma \ref{lem  kurdyka proj fixees} (see Remark \ref{rem_alpha_flat} and Lemma \ref{lem_reg_inclusion}) it is sufficient
to prove by induction on $n$ the following assertions: given
 $\alpha >0$ and $A_1, \dots, A_p$ in $\smn$,
 there exists a cell decomposition of $\R^{m+n}$
compatible with $A_1, \dots, A_p$ and
  such that for every  cell $C\subset \R^{m+n}$ of this cell decomposition satisfying $\dim C_t=n$, $(\delta C_t)_\tim$  is included in the union of no more than $2n$  definable families of empty interior that are  all $\alpha$-flat.

  For $n=0$ this is clear.
 Fix $n \in \N$ nonzero, $\alpha >0$, as well as $A_1 ,\dots, A_p$
   in $ \s _{m+n}$. Taking a cell decomposition if necessary, we can assume that the $A_i$'s are cells.  Apply Lemma \ref{lem_partition_tangents}  to all the $A_i$'s, and take a cell decomposition $\D$ of $\R^{m+n}$ compatible with all the elements of the obtained coverings. 
   Applying then the
  induction hypothesis to the elements of  $\pi_{e_{m+n}}(\D)$, we get a refinement $\D'$ of $\pi_{e_{m+n}}(\D)$. 
  
  Given a cell $D$ of $\D'$, each $A_i$ is above $D$, either the graph of a definable function, say $\xi_{i,D}$, or a band, say $(\xi_{i,D},\xi'_{i,D})$, with $\xi_{i,D}<\xi'_{i,D}$ definable functions on $D$ (or $\pm \infty$). Let $\C$ be the cell decomposition given by all the graphs $\Gamma_{\xi_{i,D}}$ and  $\Gamma_{\xi_{i,D}'}$, $i\le p$, $D\in \D'$. To check that it has the required property, fix an open cell  $C=(\xi_{i,D},\xi'_{i,D})$, with $\xi_{i,D}<\xi'_{i,D}$ definable functions on an open cell $D$ of $\D'$  (or $\pm \infty$).   Since $\D'$ is compatible with the images under  $\pi_{e_{m+n}}$ of the $\alpha$-flat sets that cover the $A_i$'s, the sets  $\Gamma_{\xi_{i,D}}$ and  $\Gamma_{\xi'_{i,D}}$ must be  
 $\alpha$-flat families, and since 
 $$\delta C_t \subset \left(\Gamma_{\xi_{i,D}}\right)_t \cup \left(\Gamma_{\xi_{i,D}'}\right)_t \cup \pi^{-1}_{e_{n}}(\delta D_t), $$
we see that the needed fact follows from the induction hypothesis.
 \end{proof}
\begin{rem}\label{rem_borne_ligne_reguliere_L_reg_dec}
We have proved a stronger statement: the distance between the regular vector $\lambda_{j(C)}$ and the tangent spaces to $\delta C_t$ can be bounded below away from zero by a positive number depending only on $n$, and not on the $A_k$'s. This is due to the fact that we apply Lemma \ref{lem  kurdyka proj fixees} with $\nu =2n$.
\end{rem}

\end{subsection}

\end{section}

\begin{section}{Regular systems of hypersurfaces}\label{sect_regular_systems}

This section is entirely devoted to the proof of Theorem \ref{thm_proj_reg_hom_pres_familles} which requires some material. We first introduce our machinery of regular systems of hypersurfaces.

 Let $Z\in \s _{n}$ and  $\lambda \in \sph^{n-1} $, with $Z \subset
 N_\lambda$ (see section \ref{sect_reg_lips} for $N_\lambda$).
If $A\in \s_n$ is the graph
  of a function $\xi:Z \to \R$  for $\lambda$,  we denote by $E(A,\lambda)$ the subset constituted by the points  which lie ``under the graph", i.e. we set:
\begin{equation}\label{eq_E_lambda}
E(A,\lambda):=\{q \in\pi_\lambda^{-1}(Z) : \, q_\lambda \leq
\xi(\pi_\lambda(q)) \}.\end{equation}

\begin{rem}\label{rem_efamilles}If $A \in \s_{m+n}$ is such that $A_t$ is the graph for $\lambda$ of a function $\xi_t:N_\lambda \to \R$ for every $t\in \R^m$, then $E(A_t,\lambda)$, $t \in \R^m$, is a definable family of sets of $\R^m \times \R^n$. Indeed, regarding $\lambda$ as an element of $\sph^{n+m-1}$ (i.e., identifying $\lambda$ with $(0_{\R^m},\lambda)$),  $E(A,\lambda)$ is also well-defined and   $E(A,\lambda)_t=E(A_t,\lambda)$, for all $t\in \R^m$.\end{rem}

\medskip

\subsection{Regular systems of hypersurfaces}
Regular systems of hypersurfaces  will help us to carry out constructions inductively on the dimension of the ambient space.

\begin{dfn}\label{dfn_systemes}
Let $B \in \s_m$. A {\bf regular system of hypersurfaces}  of $B\times \R^{n}$ (parametri\-zed by $B$) is a finite collection
$H=(H_k,\lambda_k)_{1 \leq k \leq b}$ with $b\in\N$, of definable
subsets $H_k$ of $B \times \R^{n}$ and elements  $\lambda_k \in \sph^{n-1}$ such that the
following properties hold for each $k < b $ and every $t\in B$:
\begin{enumerate}[(i)]
 \item The sets $H_{k,t}$ and $H_{k+1,t}$ are the respective
 graphs
 for $\lambda_k$ of two   functions $\xi_{k,t}:N_{\lambda_k}\to \R$ and  $\xi_{k,t}':N_{\lambda_k}\to \R$
 such that $\xi_{k,t} \leq \xi'_{k,t}$ and which are $C$-Lipschitz with $C\in \R$  independent of $t$. 
\item\label{item_reg_syst_ii} The following equality holds:
$$E(H_{k+1,t},\lambda_k)=E(H_{k+1,t},\lambda_{k+1}).$$
\end{enumerate}

Let $A\in \smn$. We
say that  $H$ is {\bf compatible}\index{compatible!regular system} with $A$, if $A
\subset \bigcup_{k=1} ^{b} H_k$. An {\bf extension}\index{extension of a regular system} of $H$ is a
 regular system of hypersurfaces (of $B \times \R^n$) compatible with the set $\bigcup_{k=1} ^b H_k$.
\end{dfn}


  Given a positive integer $k < b$, we set: $$G_k(H):=E(H_{k+1},\lambda_k)\setminus
int(E(H_{k},\lambda_{k})).$$
We shall write  $\Lambda_k(H)$ for  the  connected component  of the set $$\{ \lambda \in \sph^{n-1}: \mbox{$\lambda$ is regular for $H_k\cup H_{k+1}$}\} $$
that contains $\lambda_k$.
 
We will see   (Proposition \ref{pro_connexite_des_projections} below) that
 the  set $G_k(H)$ may be defined using any  $\lambda
\in \Lambda_k(H)$ (instead of $\lambda_k$).

 We will say
that another regular system $H'$ {\bf coincides with}\index{regular system!coincides with $H$ outside} $H$
\textbf{outside} $G_k(H)$ if for each $j$ either  $H'_j \subset
G_k(H)$ or there exists $ j'$ such that $H'_j=H_{j'}$.
\medskip

Given a regular system $H:=(H_{k},\lambda_k)_{k\le b}$ of $B\times \R^n$ and a definable set $B' \subset B$, we denote by $H_{B'}
$ the regular system of hypersurfaces $(H_{k,B'},\lambda_k)$ of $B' \times \R^n$, obtained by considering the sequence of the respective restrictions to $B'$ of the $H_k$'s (see (\ref{eq_rest_familles})). We will call it the {\bf restriction to $B'$}\index{restriction! of a regular system} of the regular system $H$.

\begin{rem}\label{rem G_k int vide}
It is always possible to assume that the $G_k(H)_t$'s are of  nonempty
interior for some $t$. Indeed, if $int(G_k(H)_t)=\emptyset$ for all $t\in B$, then $H_k=H_{k+1}$ and in
this case we may remove $(H_{k},\lambda_{k})$ from the
sequence. 
\end{rem}


\bigskip

Given a  regular system of hypersurfaces (of $B\times \R^n$, $B\in \s_m$) $H$, it will be convenient
to extend the notations in the following way. Set for any $t \in B$:
$H_{0,t}:= \{-\infty\}$ and $H_{b+1,t}:= \{+\infty\}$. By convention, all the
elements of $\sph^{n-1}$ will be regular for these two  sets. We will  also
consider that these two sets
 as the respective graphs of the two functions which take
$-\infty$ and $+\infty$ as respective constant values.
Define also $\lambda_0:=\lambda_1$, $\lambda_{b+1}:=\lambda_b $, as well as
$E(H_0,\lambda_0):= \emptyset$, $G_0(H):=E(H_1,\lambda_1)$,
$G_b(H):= (B \times \R^{n} )\setminus
 int(E(H_b,\lambda_b))$, as well as $E(H_{b+1},\lambda_{b+1}):=B\times \R^n$. Remark that now $B\times \R^{n}=\bigcup_{k=0}^b
 G_k(H)$.

\medskip

\begin{thm}\label{thm_existence_des_familles}
Let   $A \in \smn$ be such that $A_t$ has empty interior for all $t \in \R^m$. There exists a definable partition $\Pa$ of $\R^m$ such that for every $B \in \Pa$ there is a regular system of hypersurfaces of $B \times \R^{n}$ compatible with $A_{B}$.
\end{thm}

This theorem is the main ingredient of the proof of Theorem \ref{thm_proj_reg_hom_pres_familles}. 
The basic strategy of the proof of Theorem \ref{thm_existence_des_familles} (given in section \ref{sect_proof}) relies on the following observation.

\begin{pro}\label{pro_connexite_des_projections}
Let $U$ be a connected subset of $\sph^{n-1} $, $\lambda_0 \in U$,  and
let $\xi : N_{\lambda_0} \rightarrow  \R$ be a continuous definable
function.  If $U$ is regular for  $X:=\Gamma_\xi ^{\lambda_0} $ then, for each $\lambda \in U$,
the set $X$ is the graph for $\lambda$ of a function $\xi^\lambda : N_\lambda
\rightarrow \R$.
 Moreover,  $E(X,\lambda)$  is independent of $\lambda \in
U$.
\end{pro}
\begin{proof}
Let
$$C:=\{ \lambda \in U :\, \forall x \in N_\lambda, \quad card \;
\pi^{-1}_\lambda (x)\cap X =1\}.$$ We have to check that $C=U$.
Let $\lambda \in C$ and set $r(\lambda):=d(\lambda,\tau(X))$.

 We
claim  that $\bou(\lambda,\frac{r(\lambda)}{3})\cap U \subset C$.
Pick $\lambda' \in \bou(\lambda,\frac{r(\lambda)}{3})\cap U $ different
from $\lambda$ and set
$l'=\pi_\lambda(\lambda')$. We have to show that the line $L$
 generated by $\lambda'$ and passing through any $x \in N_\lambda$ intersects $X$ in exactly
 one point.
The line $L$ is the graph for $\lambda$ of a function
$\eta(x+tl')=\alpha \cdot t$ with $\alpha > \frac{2}{r(\lambda)}$ (we assume $\alpha>0$ for simplicity).

Since $\lambda \in C$, the set $X$ is the graph
for $\lambda$ of a function $\xi^{\lambda}$. By definition of $r(\lambda)$, $|d_x \xi^\lambda|$ (which exists almost everywhere) is bounded by $\frac{2}{r(\lambda)}$. It easily follows from the Mean Value Theorem that $\xi^\lambda$ is $\frac{2}{r(\lambda)}$-Lipschitz.
 
  This implies that for
$t$ positive large enough we will have $\eta(x+tl') \geq \xi^{\lambda}(x+tl')$ and
 $\eta(x-tl') \leq \xi^{\lambda}(x-tl')$ (since $\eta$ is growing faster than $\xi^\lambda$). Thus, there is a point $q\in \pi_{\lambda}(L)$ such
that $\xi^{\lambda}(q)=\eta(q)$, which implies that the line $L$ cuts $X$.
Uniqueness of the intersection point is clear from the fact that one function is growing faster than the other.
This yields   that $\bou(\lambda,\frac{r(\lambda)}{3}) \cap U \subset C$.

We have shown that $C$ is open in $U$. Let us now show that it is also closed in $U$. Consider $\lambda \in U$ and a
continuous definable  arc $\gamma$ in $C$ tending to $\lambda$.
Since $r(\gamma(t))$ tends to $r(\lambda)$ which is nonzero, the
ball $\bou(\gamma(t),\frac{r(\gamma(t))}{3})$ contains $\lambda$ for $t>0$ small
enough. The closeness of $C$ therefore follows from the fact that $\bou(\gamma(t),\frac{r(\gamma(t))}{3})\cap U  \subset C$. As $U$ is connected, this yields $U=C$.

It remains to check that $E(X, \lambda)$ is independent of
$\lambda \in U$. It is the closure of one of the two connected
components of the complement of  $X$.  The set $X$ is the zero locus of the
function $f(q)=q_{\lambda_0}-\xi(\pi_{\lambda_0}(q))$. Locally, at a smooth point $q$ of
$X$ it is clear that $E(X,\lambda)$ is determined by the sign of $d_q f
(\lambda)$.
  But as $U$ is regular for $X$,  this function never vanishes, and consequently $E(X,\lambda)$
 is  independent of $\lambda \in U$.
\end{proof}

\medskip

Given $e \in \sph^{n-1}$, we define a mapping $\widetilde{\pi}_e : \sph^{n-1} \setminus \{\pm e \} \to \sph^{n-1}\cap N_{e}$ by setting 
\begin{equation}\label{eq_pitilda}\widetilde{\pi}_e (u):= \frac{\pi_e (u)}{|\pi_e (u)|}.\end{equation}

 \begin{rem}\label{rem_connexite_pi_tilda}
Let $e \in \sph^{n-1}$ and suppose $\lambda_0\in \sph^{n-1}\cap N_e$ to
   be  regular for a subset $A \subset N_e$.
  Since the elements of $\widetilde{\pi } _{e}^{-1} (\lambda_0)$ lie above the line generated by $\lambda_0$, for each positive real number $a$,  the set
     $$C:=\widetilde{\pi } _{e}^{-1} (\lambda_0) \cap \{\lambda
     \in \sph^{n-1} : d(\lambda, \{\pm e\}) \geq a\}$$
  is regular for $\pi_{e} ^{-1} (A)$.  
Moreover, by Proposition \ref{pro_connexite_des_projections}, if $A$ is the
  graph of a Lipschitz function for $\lambda_0$
  then $\pi_{e}^{-1}
  (A)$ is the graph of a Lipschitz function for each
  $\lambda \in  C$.
Furthermore, the latter proposition also entails that in this case we have  for all
  $\lambda \in  C$:
  $$E(\pi_{e} ^{-1} (A),\lambda)=\pi_{e} ^{-1} (E(A,\lambda_0)).$$
\end{rem}

\subsection{Two preliminary Lemmas}
Proving Theorem
\ref{thm_existence_des_familles} requires two preliminary lemmas on regular systems of hypersurfaces.
The first one will make it possible for us to assume that the
interiors of the  $G_k(H)_t$'s are connected.

\begin{lem}\label{lem connexite des G_k}
Let $H$ be a regular system of hypersurfaces of $B \times \R^n$, $B \in \s_m$. There exists a definable partition $\Pa$ of $B$ such that for every $B' \in \Pa$, we can find an
extension $\widehat{H}$  of $H_{B'}$ such that the set $int(
G_k(\widehat{H} )_t)$ is connected for all $t \in B'$ and all $k$.
\end{lem}
\begin{proof}
Let $1 \leq k \leq b-1$ and suppose that there is $t$ for which $int(G_k(H)_t)$ is not connected.  Applying Remark \ref{rem_cc} to  $int(G_k(H))$ provides a partition $\Pa$ of $B$.  Let $B'\in \Pa$. Possibly replacing $H$ with $H_{B'}$, we see that we can assume that the the property displayed in Remark \ref{rem_cc} holds for  $int(G_k(H))$.

 Let
 $A_1, \dots, A_\nu$ be the connected components of $int(G_k(H))$.
 Set
$A'_i=\pi_{\lambda_k}(A_i)$ for $i\le \nu$. For $t \in B'$,  every fiber  $A_{i,t}$ is of the form:
$$ \{ q
 \in A_{i,t}' \oplus  \R\cdot \lambda_k \; :\;\, \xi_{k,t} (\pi_{\lambda_k}(q))< q_{\lambda_k}
 < \xi'_{k,t} (\pi_{\lambda_k}(q))\}. $$
 Clearly $\xi_{k,t}=\xi'_{k,t}$ on the boundary of $ A'_{i,t}$. We thus may
 define some Lipschitz functions $\eta_i$, $1 \leq i \leq \nu
 -1$, as follows. We set over $A'_{j,t}$,  $\eta_{i,t}:= \xi'_{k,t}$,
 when $1 \leq j \leq i$, and $\eta_{i,t}: =\xi_{k,t}$  whenever $i<j$. Extend the function $\eta_{i,t}$ by setting $\eta_{i,t}:= \xi_{k,t}=\xi_{k,t}'$ on $N_{\lambda_k} \setminus \pi_{\lambda_k}(int(G_k(H)))$.

 Since $\eta_{1,t} \leq \dots \leq \eta_{(\nu-1),t}$,
 it suffices to
 \begin{itemize}
\item let $\widehat{H}_j:=H_j$ and
$\widehat{\lambda}_j:=\lambda_j$ if $j \leq k$
 \item let $\widehat{H}_{j,t}$ be the graph of $\eta_{j-k,t}$ for
 $\lambda_k$ (for every $t\in \R^m$) and $\widehat{\lambda}_j:=\lambda_k$
 for $ k +1 \leq j \leq k+\nu-1 $
\item let $\widehat{H}_j:=H_{j-\nu+1}$ and
$\widehat{\lambda}_j:=\lambda_{j-\nu+1}$ if $ k +\nu \leq j \leq b+\nu-1$.
\end{itemize}
This is clearly a regular system of hypersurfaces. Note that the
$int(G_j(\widehat{H}))$, $k\leq j <k+\nu$, are the connected components of $int(G_k(H))$.
\end{proof}

\begin{lem}\label{lemme_extension_des_familles}
Let $H =\hk$ be a regular system  of hypersurfaces of $B \times \R^n$, $B \in \s_m$, and
let $j \in \{1,\dots, b\}$.
Let $X$ be a definable subset  of $G_j(H)$  and assume that
$\lambda_j$ is regular for $X$.
 Then
$H$ can be extended to a regular system  of hypersurfaces
$H'$ compatible with $X$ and which coincides with $H$
outside $G_j(H)$.
\end{lem}

\begin{proof}
By property $(i)$ of Definition \ref{dfn_systemes}, for every parameter $t\in B$, the sets $H_{j,t}$ and $H_{j+1,t}$
 are the respective graphs for $\lambda_j$ of two
  functions $\xi_{j,t}$ and $\xi'_{j,t}$.
By Propositions \ref{pro_proj_reg_decomposition_en_graphes} and \ref{pro_extension_fonction_lipschitz},  the definable family  $X_t$, $t \in B$, may be included in a finite
number of graphs for $\lambda_j$ of definable families of functions on $N_{\lambda_j}$,  say
$\theta_{1,t},\dots,\theta_{\nu,t}$,  $C$-Lipschitz for every $t\in B$ with $C\in \R$ independent of $t$.
Furthermore, by Proposition \ref{pro_famille_lipschitz_ordonnees}, these families of functions can be assumed to be totally ordered (for relation $\le$), and
satisfy  $\xi_{j,t} \leq \theta_{i,t} \leq \xi'_{j,t}$, for every $t$.
Now,
\begin{itemize}
\item let  $H'_k:=H_k$  and $\lambda'_k:=\lambda _k$
 whenever $1 \leq k \leq  j$,
\item let $H'_{k,t}$ be the graph of $\theta_{k-j,t}$ for
$\lambda_j$ and $\lambda '_{k} :=\lambda_j$ for $j < k \leq j+\nu $, $t \in\R$,
\item let   $H'_{k}:=H_{k-\nu}$ and
$\lambda'_{k}:=\lambda_{k-\nu}$, whenever $j+ 1 +\nu \leq k \leq b
+\nu$.
\end{itemize}
Properties  $(i)$ and $(ii)$ of Definition \ref{dfn_systemes} clearly hold by construction.
\end{proof}

\begin{rem}\label{rem_lem_lignes}
 It will be of service that, in the proof of the above lemma, no extra regular vector was necessary, i.e. $\{\lambda_1,\dots,\lambda_b\}=\{\lambda'_1,\dots,\lambda_{b+\nu}'\}$.  
\end{rem}

\subsection{Proof of Theorem \ref{thm_existence_des_familles}}\label{sect_proof}
The proof is divided into four steps.  The strategy is to rely on Lemma \ref{lemme_extension_des_familles}.
The reader is invited to first glance at step \ref{step4}, which was deliberately made very short and sheds light on the reason why this lemma is helpful. The problem to get a regular system of hypersufaces compatible with a set $A$ using this lemma is that it requires to already have a regular system $H=\hk$ such that $\lambda_j$ is regular for $A\cap G_j(H)$ (for all $j$). This fact is not granted by the system $H$ provided by step \ref{step1}, which satisfies a slightly weaker property. We therefore shall provide (in step \ref{step2}) another system $\widehat{H}$ (see the paragraph just before step \ref{step2} for more details on this issue) and then  construct in step \ref{step3} what can be considered as a ``common refinement'' of $H$ and $\widehat{H}$, which will be satisfying to apply   Lemma \ref{lemme_extension_des_familles} in step \ref{step4}.

Let $A \in \smn$ be such that $A_t$ has empty interior for every $t\in \R^m$. Let also $\eta  \in (0,1]$ and $\lambda \in \sph^{n-1}$.
 We are actually going to prove  by induction on $n$ that, given any such $A$, $\lambda$, and $\eta$, there
exists a definable partition $\Pa$ of $\R^m$  such that for every $B \in \Pa$ we can find a regular system  of hypersurfaces of $B \times \R^{n}$ compatible
with   $A_{B}$  and such that all the $\lambda_k$'s (see Definition \ref{dfn_systemes}) can
be chosen  in   $\bou(\lambda,\eta)\cap \sph^ {n-1}
 $. 

Since the result is clear for $n=1$, we take $n\ge 2$ and assume  it to be true for
$(n-1)$. Observe that it is enough to establish the claimed statement for arbitrarily small values of $\eta$. 
 As explained just above, we split the induction step into $4$ steps.

\begin{step}\label{step1}
There exists a definable  partition $\Pa$ of $\R^m$  such that for every $B \in \Pa$, there is a regular system of hypersurfaces $H=\hk$ of $B \times \R^n$,  with $\lambda_k \in \sph^{n-1}\cap \bou(\lambda,\frac{\eta}{2})$,  such that for every $k$ the set $int(G_k(H))\cap A_{B}$ has a regular vector $\mu \in \sph^{n-1}\setminus \bou(\pm \lambda , \eta)$.
\end{step}
\begin{proof}[Proof of step \ref{step1}.]
Take  $e \in \sph^{n-1}$ such that $\pm e\notin \bou(\lambda ,
\eta)$.
 By Lemma \ref{lem_partition_tangents}, 
  for each $\sigma >0$, there are finitely many $(m,\frac{\sigma}{2})$-flat sets $U_1,\dots,U_\omega$  that cover $A$.
Consider such a
 partition for  $\sigma=t_\nu$, where $t_\nu$ is given by
 Lemma \ref{lem_kurdyka},  with $\nu$ equal to the maximal number of connected components of $\pi_e^{-1}(x)\cap A_t$, $(t,x)\in \R^m\times N_e$.   Changing $\eta$, we may assume that $\eta\leq \frac{t_\nu}{4}$.
 
 Take a cell decomposition of $\R^m\times N_e$ (identify $\R^m \times N_e$ with $\R^m \times \R^{n-1}$) which is compatible with the $\pi_e(U_i)$, $i\le \omega$, and denote by
  $(W_i)_{i\in I}$ the collection of the cells of this cell decomposition that are open  (in $\R^m \times N_e$).
  
 Since the set $A_t$ has empty interior for each $\tim$,   the set $A_t\cap \pi_e ^{-1}(W_{i,t})$ is (for each $i\in I$ and $t$), above $W_{i,t}$, the union of at most $\nu$ (possibly $0$) graphs for $e$ of continuous functions (not necessarily Lipschitz).  
  
Choose  $\eta' >0$ such that we have in $\sph^{n-1}\cap N_e$:
\begin{equation}\label{rayon eta prime}
\bou(\widetilde{\pi}_e(\lambda) ,\eta') \subset
\widetilde{\pi}_e(\bou(\lambda ,\frac{\eta}{2})).
\end{equation}

  Apply the
 induction hypothesis (identify $\R^m \times N_e$ with $\R^m \times \R^{n-1}$)
  to the families $(\delta W_{i,t})_{t \in \R^m}$ to get a partition $\Pa$. Fix $B \in \Pa$. There is a  regular system of  $B \times \R^{n-1}$, $\overline{H}=(\overline{H}_k,\overline{\lambda}_k)_{k\leq b}$, such that all the
$\overline{\lambda}_k$'s belong to
 $\bou(\widetilde{\pi}_e(\lambda) ,\eta') $.

By Lemma \ref{lem connexite des
 G_k}, up to a refinement of the partition, we may assume that each $int(G_k(\overline{H})_t)$ is connected for all $t \in B$. We may also assume these sets to be of
 nonempty interior for some $t$ (see Remark \ref{rem G_k int vide}).  Up to an extra refinement, we may assume that it happens for all $t\in B$ (by Remark \ref{rem_cc}). 

  We claim
  that for each $j$ and  $k$ and for every $t$, either $int(G_k(\overline{H})_t)$ is disjoint from $W_{j,t}$
 or
 $int (G_k(\overline{H})_t)\subset W_{j,t}$.
 To see this, observe that, as $\overline{H}$ is compatible with the $\delta W_{j,t}$'s, all the sets   $W_{j,t} \cap int(G_k(\overline{H})_t)$ are open
 and of empty (topological) boundary in $int(G_k(\overline{H})_t)$, for each $t$. Hence, if nonempty,  these are  connected
 components of $int(G_k(\overline{H})_t)$. But, as $int(G_k(\overline{H})_t)$ is connected, this entails that $W_{j,t} \cap int(G_k(\overline{H})_t)$ is either the
 empty set  or $int(G_k(\overline{H})_t)$ itself, as claimed. 

As the $W_{i,t}$'s are disjoint from each other,  for each $k$ there is a unique $i$ such that $int(G_k(H)_t)\subset W_{i,t}$. After a possible refinement of the partition $\Pa$, we can assume that for each $k$, $int(G_k(H)_t)$ meets the same $W_{i,t}$ for all $t$, i.e. that for every $B\in \Pa$,  $i$ depends only on $k$ and not on $t\in B$.

We turn to define the  regular  system  $H$ claimed in step \ref{step1}. For $1\leq k\leq b$, let:
$$ H_k := \pi_e^{-1}(\overline{H}_k).$$

Since $\overline{\lambda}_k \in \bou(\widetilde{\pi}_e(\lambda) ,\eta')$, by (\ref{rayon eta prime}), we have $\overline{\lambda}_k \in
 \widetilde{\pi}_e (\bou(\lambda,\frac{\eta}{2}))$.
 Choose some
 $\lambda_k  \in  \widetilde{\pi}_e ^{-1}(\overline{\lambda}_k) \cap
 \bou(\lambda,\frac{\eta}{2})$.

As $\lambda_k \in \bou(\lambda,\frac{\eta}{2})$ for all $k$ and neither $e$ nor $-e$
belongs to $\bou(\lambda,\eta)$ we have: $$d(\lambda_k , \pm e)
\geq \frac{\eta}{2}, \qquad \forall\, k \leq b.$$
As a matter of fact, by Remark \ref{rem_connexite_pi_tilda},
 as  $\overline{H}$ fulfills conditions $(i)$ and $(ii)$ of Definition \ref{dfn_systemes}, these conditions  are also fulfilled by $H:=(H_k,\lambda_k)_{k\leq b}$.

 By Lemma \ref{lem_kurdyka} and our choice of
$\sigma$, for all $k$, the set $A_B\cap int(G_{k}(H))$ is the union of finitely many definable sets having a common regular element
$\mu \in \sph^{n-1}$ (since we have seen that each $int(G_k(\overline{H})_t)$ is
included in $W_{j,t}$ for some $j$ independent of $t \in B$). Moving slightly $\mu$, we may assume that $d(\mu,\pm \lambda) \geq \eta$ (we have assumed $\eta \leq \frac{t_\nu}{4}$). This completes the proof of the first step.\end{proof}

\medskip

The desired partition of $\R^m$ will be obtained after finitely many refinements of the partition $\Pa$. Clearly, it is enough to prove the result for all the sets $A_{B}$, $B \in \Pa$. We thus can fix $B$ in $\Pa$ and  identify $A$ and $A_{B}$ in the next steps below.  For simplicity, we will not indicate either the partitions of the parameter space $\R^m$ resulting from the successive refinements of $\Pa$, working always  with $A$ (instead of $A_{B}$).

\bigskip

The flaw of the first step is that the regular vector $\mu$ that we get for $G_k(H)\cap A$  might not be in $\Lambda_k(H)$. If it belongs to this set, Proposition \ref{pro_connexite_des_projections} and Lemma \ref{lemme_extension_des_familles} suffice to conclude (see step \ref{step4}). Had the vector $e$ (used in step \ref{step1}) been regular for $A$, we could have required $\mu \in \Lambda_k(H)$ in step \ref{step1}, using
Lemma \ref{lem_direction_generique_droit_proj} below in the same way as we will do in step \ref{step2} to construct $\widehat{H}$ by means of $\pi_\mu$ (see assumption (\ref{hypothese key lemma})). One could therefore think that not much was achieved so far as we need a regular vector to carry out our construction and finding a regular vector is all our purpose. However, since we can focus on the sets $A\cap G_p(H)$, which all have a regular vector (provided by step \ref{step1}), by repeating in step \ref{step2} the construction of the first step (replacing $e$ with $\mu$ and making use of Lemma \ref{lem_direction_generique_droit_proj}), we will get a system $(\widehat{H}_k,\widehat{\lambda}_k)_{k\le \widehat{b}}$ with $\widehat{\lambda}_k \in\Lambda_p (H)$ regular for $G_p(H)\cap G_k(\widehat{H})\cap A$, for each fixed $p\le b$.  
 It will then remain to find (in  step \ref{step3}, see the paragraph before step \ref{step3} for more details) a common extension of $H$ and $\widehat{H}$, obtained at step \ref{step1} and \ref{step2} respectively.

\medskip

\begin{step}\label{step2}
Fix $p\leq b$.   There exists a regular system of hypersurfaces $\widehat{H}=(\widehat{H}_k,\widehat{\lambda}_k)_{k\le \widehat{b}}$ such that for every $k $,  $\widehat{\lambda}_k \in\Lambda_p (H)\cap  \bou(\lambda,\eta)$  and is  regular for $G_p(H)\cap G_k(\widehat{H})\cap A$.
\end{step}
\begin{proof}[Proof of step \ref{step2}.]
 Note that as $\lambda _p$ is regular for the set $H_p
\cup H_{p+1}$,  there exists $r >0$
such that $\bou(\lambda_p,r)$ is regular for $H_p \cup H_{p+1}$. Taking $r$ smaller if necessary, we may assume that $r\leq \frac{\eta }{4}$.

Let $r' >0$ be such that we have in $\sph^{n-1}\cap N_\mu$:
\begin{equation}\label{eq_def_de_r'}
\bou(\widetilde{\pi}_\mu(\lambda_p),r') \subset
\widetilde{\pi}_\mu(\bou(\lambda_p,\frac{r}{2})).
\end{equation}

 To complete the proof of step \ref{step2}, we need a lemma.
\begin{lem}\label{lem_direction_generique_droit_proj}
Let $l$ in $\sph^{n-1}$, $0<r \le 1$, and $\kappa \in \N$. Let  $C$  be a
 subset of $\G^n$ and $\mu\in \sph^{n-1}$ such
that
 \begin{equation}\label{hypothese key lemma}
 d(\mu,C) >0.
  \end{equation}
 There exists $\alpha >0$ such  that
for any $P_1,\dots,P_{\kappa}$ in $C$ and any $y \in
\widetilde{\pi}_\mu( \bou(l,\frac{r}{2}))$  there exists
$\widehat{\lambda} \in \bou(l,r) \cap
\widetilde{\pi}_\mu^{-1}(y)$ such that:
$$d(\, \widehat{\lambda} , \bigcup_{i=1} ^{\kappa} P_i) \geq \alpha.$$
\end{lem}

The proof of this lemma is postponed. We first see why it is
enough to carry out the proof of step \ref{step2}. Let $\kappa\ge 1$ be the maximal number of connected components of  $A \cap G_p(H)\cap \pi_\mu^{-1}(x)$, $x\in N_\mu$.
Applying this lemma with this integer $\kappa$, with 
 $C =\bigcup_{t \in B} \,\tau (A_t \cap G_p(H)_t)$ and
$l =\lambda _p$ ($\mu$ being given by step $1$),
we get a positive constant $\alpha$.

By Lemma \ref{lem_partition_tangents}, we can cover $ (G_p(H)_t \cap A_t)_{\tim}$ by $\frac{\alpha}{2}$-flat families, say $U'_{1,t},\dots,U'_{\omega',t}$, $\omega'\in \N$.  
 Take a cell decomposition $(W_i')_{i\in I'}$ of $\R^m\times N_\mu$ (identify $\R^m \times N_\mu$ with $\R^m \times \R^{n-1}$) which is compatible with the $\pi_\mu(U_i')$, $i\le \omega'$.

 For any  $i\in I'$ the family
$\pi_\mu^{-1} ( W_{i,t }') \cap G_p(H)_t \cap A_t$, $t\in B$, is thus included in the union of no more than $\kappa$ $\frac{\alpha}{2}$-flat families.

 Lemma \ref{lem_direction_generique_droit_proj} thus implies that  for any  $i\in I'$ and any $y \in
\widetilde{\pi}_\mu(\bou(\lambda _p , \frac{ r}{2}))$, there exists
$\widehat{\lambda} \in \bou(\lambda _p,r) \cap
\widetilde{\pi}_\mu ^{-1}(y)$ such that for any $t \in B$:
\begin{equation}\label{existence de lamda chapeau}
d\big(\widehat{\lambda}\, ,\, \tau(\pi^{-1} _\mu(W'_{i,t}
) \cap G_p(H)_t \cap A_t) \big) \geq \frac{\alpha}{2}.
\end{equation}
 Apply the induction hypothesis to get a regular system
 of hypersurfaces $H''=(H_k'',\lambda_k'')_{k\le b''}$ of $B \times N_\mu$ (identify $N_\mu$ with $\R^{n-1}$, up to a  refinement of the partition $\Pa$)
  compatible with  the  $\delta W'_{i,t}$'s.
Do it in
 such a way that all the  $\lambda''_k$
 are elements  of  $\bou(\widetilde{\pi}_\mu(\lambda _p),r')$ (where $r'$ is given by (\ref{eq_def_de_r'})).

 Define now:
\begin{equation}\label{def de H chapeau}
 \widehat{H}_{k,t} := \pi_\mu ^{-1}(H''_{k,t}).
\end{equation}

 The compatibility with  the families $\delta W'_{i,t}$ implies that every $int(G_k(H''))_t$
 is included in  $W'_{i,t}$ for some $i$
 (possibly refining the partition of the parameter space), by the same argument as the one  we used in step \ref{step1} for $G_k(H)$ and the partition
 $(W_i)_{i \in I}$.

As a matter of fact, according to (\ref{existence de lamda
 chapeau}) for   $y=\lambda_k ''$, we know  that for every integer  $ k\leq b''$
  there exists
 $\widehat{\lambda} _k  \in \bou(\lambda_p,r) \cap
\widetilde{\pi}_\mu ^{-1}(\lambda'' _k)$ such that for any $t\in B$:
\begin{equation}\label{existence de lamda chapeau par raport a  G''}
d\big(\, \widehat{\lambda}_k \, ,\, \tau(\pi^{-1}
_\mu(int(G_k(H'')_t)) \cap G_p (H)_t\cap A_t) \big) \geq \frac{\alpha}{2}.
\end{equation}

Let us check that  $\widehat{H}:=(\widehat{H}_k,\widehat{\lambda}_k)_{k \leq \widehat{b}}$ (where $\widehat{b}:=b'')$ is the desired   regular system of hypersurfaces.
For this purpose, observe  that, since neither  $\mu$ nor $-\mu$  belongs to $\bou(\lambda ,\eta)$, we have  for each $k$ (recall that $r \leq \frac{\eta}{4}$ and $\lambda_p\in \bou(\lambda,\frac{\eta}{2})$, as well as  $\widehat{\lambda} _k  \in \bou(\lambda_p,r)$):
$$d(\widehat{\lambda}_k, \pm \mu) \geq r.$$

 By Remark \ref{rem_connexite_pi_tilda}, as $\widehat{\lambda}_k \in \widetilde{\pi}_\mu
 ^{-1}(\lambda_k'')$,
 this implies that the family $\widehat{H}$  fulfills the conditions of Definition \ref{dfn_systemes}.

 Furthermore, as $\bou(\lambda_p ,r) \subset
 \bou(\lambda,\eta)$ (since $r \leq \frac{\eta}{4}$ and $\lambda_p\in \bou(\lambda,\frac{\eta}{2})$), all the $\widehat{\lambda}_k$'s belong to
 $\bou(\lambda,\eta)$. Note also that as $\bou(\lambda_p,r)$ is regular for $H_p \cup H_{p+1}$, the vector
   $\widehat{\lambda} _k$ belongs to $\Lambda_p(H)$. This completes the proof of the second step. \end{proof}

\bigskip

The inconvenience of step \ref{step2} (we would like to apply Lemma \ref{lemme_extension_des_familles}, see step \ref{step4}) is that the provided vector is regular for  $A\cap G_p(H)\cap G_k(\widehat{H})$ (instead of $A\cap G_k(\widehat{H})$).  If $\widehat{H}$ were an extension of the family $H$ constructed in step \ref{step1}, this would be no problem since in this case we would have $G_k (\widehat{H})\subset G_p(H)$ (or $int(G_k(\widehat{H})) \cap int(G_p(H))=\emptyset$). Thus, we will have to find a common extension $\widetilde{H}$ of $H$ and $\widehat{H}$ given by steps \ref{step1} and \ref{step2} respectively. This is what is carried out in the proof of step \ref{step3}.

\medskip

\begin{step}\label{step3}Fix  $p\le b$.
 There exists an extension $\widetilde{H}=(\widetilde{H}_k,\widetilde{\lambda}_k)_{k \leq \widetilde{b}}$ of $H$  which  coincides with $H$ outside $G_p(H)$, and such that for each $k$, $\widetilde{\lambda}_k$ belongs to  $\bou(\lambda,\eta)$ and is regular for  $A \cap G_k(\widetilde{H})\cap G_p(H)$.
\end{step}\begin{proof}[Proof of step \ref{step3}.]  Let $\widehat{H}$ be the regular system given by step \ref{step2} and let $k \leq \widehat{b}$ be an integer. Because $\widehat{\lambda}_k
\in \Lambda_p(H)$, by Proposition \ref{pro_connexite_des_projections},  the sets $H_{p}$ and $H_{p+1}$ are respectively the
graphs for $\widehat{\lambda}_k$ of two
functions $\zeta_{k}$ and $\zeta'_{k}$. Moreover, the set $\widehat{H}_{k}$
is also the graph for  $\widehat{\lambda}_k$ of a
 function $\widehat{\xi}_{k}$. Define:
$$\theta_{k} := \min ( \max( \zeta_{k}, \widehat{\xi}_{k}),\zeta'_{k})$$
in order to get a function whose  graph for $\widehat{\lambda}_k$ is included in $G_p(H)$.
We now define the desired regular  family
 $(\widetilde{H}_k,\widetilde{\lambda}_k)_{1\leq k \leq \widetilde{b}} $ as
follows:

\begin{itemize}
 \item Let  $\widetilde{H}_k:=H_k$  and $\widetilde{\lambda}_k
:=\lambda_k$ if $k < p$.
 \item Let $\widetilde{H}_p :=H_p$ and
$\widetilde{\lambda}_p:= \widehat{\lambda}_1$.
 \item  Let
$\widetilde{H}_{k}$  be the graph of $\theta_{k-p}$ for
$\widehat{\lambda}_{k-p}$, and let $\widetilde{\lambda}_k
:=\widehat{\lambda}_{k-p}$, whenever $ p+1 \leq k \leq
p+\widehat{b}$.
 \item And finally let
$\widetilde{H}_k:=H_{k-\widehat{b}} \: $ and $\:
\widetilde{\lambda}_k:=\lambda_{k-\widehat{b}} \: $ if $
\:p+\widehat{b} +1 \leq k \leq b+\widehat{b} $.
\end{itemize}

Let us check that properties $(i)$ and $(ii)$ of Definition \ref{dfn_systemes} hold for the
family $\widetilde{H}$.
For $k < p-1$, or $k \geq p+\widehat{b}+1$, the result is clear
since the family $\widetilde{H}$ is indeed the family $H$ (after renumbering).

For
$k=p-1$, properties  $(i)$ and $(ii)$ for  $\widetilde{H}$ follow from $(i)$ and $(ii)$ for $H$  and Proposition \ref{pro_connexite_des_projections} since we have assumed $\widehat{\lambda}_1 \in \Lambda_p(H)$.


It remains to check $(i)$ and $(ii)$  for $\widetilde{H}_{k+p}$, with $0 \le k \leq \widehat{b}$. We start with $(i)$.
 By $(i)$ for $\widehat{H}$,
the set $\widehat{H}_{k+1}$ is  the graph for
$\widehat{\lambda}_k$ of  a function
 $\widehat{\xi}'_{k}$ such that   $\widehat{\xi}_{k}\leq
 \widehat{\xi}'_{k}$. Define now: $$\theta'_{k}=\min(\max(\zeta_{k},\widehat{\xi}'_{k}),\zeta'_{k}).$$

 \noindent {\bf Claim.} The graph of $ \theta'_{k}$ for
 $\widehat{\lambda}_k$  is  that of  $\theta_{k+1}$ for
 $\widehat{\lambda}_{k+1}$.

 To see this, note that the graph of $
 \theta'_k$ (resp. $\theta _{k+1}$) for $\widehat{\lambda}_k$ (resp. $\widehat{\lambda}_{k+1}$)  matches with  $\widehat{H}_{k+1}$ over  $E(H_{p+1},\widehat{\lambda}_k) \setminus  E(H_p,\widehat{\lambda}_k)$  (resp. $\widehat{\lambda}_{k+1}$).  But, by Proposition \ref{pro_connexite_des_projections},  the sets
 $E(H_p,l)$ and  $E(H_{p+1},l)$ do not depend on $l \in \Lambda_p(H)$.
As $\widehat{\lambda}_k$ and $\widehat{\lambda}_{k+1}$ both belong
to $ \Lambda_p(H)$, this already shows that the two graphs involved in the above claim match over $int(G_p(H))$.

 The graph of $ \theta'_{k}$ (resp. $\theta _{k+1}$) for
 $\widehat{\lambda}_k$ (resp. $\widehat{\lambda}_{k+1}$)  is also constituted by the points of $H_p \setminus int ( E(\widehat{H}_{k+1},\widehat{\lambda}_{k}))$ (resp. $\widehat{\lambda}_{k+1}$) on the one hand and by  the points of $H_{p+1}\cap E(\widehat{H}_{k+1},\widehat{\lambda}_{k})$ (resp. $\widehat{\lambda}_{k+1}$) on the other hand. By $(ii)$ for $\widehat{H}$, the claim ensues.

This claim proves that $\widetilde{H}_{p+k+1}$, which is by definition the graph   of  $\theta_{k+1}$ for
 $\widehat{\lambda}_{k+1}$, is indeed also the graph of
$\theta'_{k}$ for $\widehat{\lambda}_k$. Therefore,  to check
$(i)$  (for $\widetilde{H}_{k+p}$, $k\leq \widehat{b}$), we just have to prove that $\theta_{k} \leq \theta'_{k}$.
But, as $\widehat{\xi}_{k} \leq \widehat{\xi}'_{k}$, this comes down from the respective  definitions of $\theta'_k $ and
 $\theta_{k}$. 

Let us  check property $(ii)$ for $\widetilde{H}_{k+p}$, for $k\leq \widehat{b}$. Observe first that if $k=\widehat{b}$,  it is then
a consequence of Proposition~\ref{pro_connexite_des_projections}, since
we have assumed that $\widehat{\lambda}_k$ belongs to
$\Lambda_p(H)$.

 Let now $k$ be such that $0 \leq k \leq \widehat{b}-1$. First note
 that by $(ii)$ for $\widehat{H}$ we have:
 $$E(\widehat{H}_{k+1},\widehat{\lambda}_k)=
 E(\widehat{H}_{k+1},\widehat{\lambda}_{k+1}).$$

But,  $E(\widetilde{H}_{k+p+1},\widehat{\lambda}_k)$ (resp. $\widehat{\lambda}_{k+1}$)
is  constituted by the points of $E(H_p,\widehat{\lambda}_k)$ (resp. $\widehat{\lambda}_{k+1}$) together with the points of  $E(H_{p+1},\widehat{\lambda}_k)\cap E(\widehat{H}_{k+1},\widehat{\lambda}_k)$ (resp. $\widehat{\lambda}_{k+1}$).
As   $\widehat{\lambda}_{k+1}$  and $\widehat{\lambda}_k$ both
belong to  $\Lambda_p(H)$, this establishes $(ii)$ for $\widetilde{H}$.

To complete the proof of  step \ref{step3}, it remains to make sure that for every $k\leq \hat{b}$ the vector $\widetilde{\lambda}_{k+p}$  is regular for $G_{k+p}(\widetilde{H}) \cap G_p(H) \cap A$. By definition, we have  $\widetilde{\lambda}_p=\widehat{\lambda}_1$,
   $\widetilde{\lambda}_{k+p}=\widehat{\lambda}_k$ for $1\le k\le \widehat{b}$,  and:
\begin{equation}\label{eq incl gtilde}
G_{k+p} (\widetilde{H})\subset  G_k(\widehat{H}) \cap G_p(H),
\end{equation}
for each $0 \leq k \leq \widehat{b}$.

As for any $k$ the vector $\widehat{\lambda}_k$ is regular for $A\cap G_k (\widehat{H}) \cap G_p(H)$ (see step \ref{step2}),
 this implies that for each $ k \leq \widehat{b}$,
 the vector $\widetilde{\lambda}_{k+p}$ is regular for
 $A\cap G_{k+p}(\widetilde{H}) $. This completes the proof of the third step.\end{proof}

\bigskip

\begin{step}\label{step4}
There is a regular system of hypersurfaces $\check{H}=(\check{H}_k,\check{\lambda}_k)_{k \leq \check{b}} $ compatible with $A$ and such that  $\check{\lambda}_k\in\bou(\lambda,\eta)$ for all $k$. 
\end{step}
\begin{proof}[Proof of step \ref{step4}]
 By Lemma \ref{lemme_extension_des_familles} (applied $(\widetilde{b}+1)$ times to  $\widetilde{H}$  of step \ref{step3}), we may extend
 $\widetilde{H}$  to a regular system $\check{H}=(\check{H}_k,\check{\lambda}_k)_{k \leq \check{b}} $ compatible with  the set
 $$ G_p(H) \cap \bigcup_{k=0}^{\widetilde{b}}G_{k}(\widetilde{H})
 \cap A=G_p(H)\cap A.$$
By  step \ref{step3}, we have  $\check{\lambda}_k\in\bou(\lambda,\eta)$ for all $k$ (see Remark \ref{rem_lem_lignes}). 
 Since $\check{H}$ is an extension of $H$ ($\widetilde{H}$ being itself an extension of $H$) which coincides with $H$ outside
 $G_p(H)$, we may carry out this construction  on all the $G_p(H)$'s  successively.
 This provides the desired regular system.
\end{proof}

\medskip

It remains to prove Lemma \ref{lem_direction_generique_droit_proj}.
 The
 lemma below describes a
property of $\widetilde{\pi}_\mu$ that we need for this purpose.

\begin{lem}\label{lem trans aux fibres pi P tilda}
Let $\mu \in \sph^{n-1}$, $T \in \G^n$, and $x \in T 
$. If $v \in \sph^{n-1}$ is tangent  at $x$ to the curve $\widetilde{\pi}^{-1} _\mu ( \widetilde{\pi} _\mu(x))$ then we have:
$$d(\mu,T) \leq d (v, T).$$
\end{lem}
\begin{proof}
Let $w$ be the vector of $ T$ which realizes $d(v,
T)$. Remark that the vectors $x$, $\mu$, and  $v$ are in the same
two dimensional vector space. Moreover  $(x,v)$ is an orthonormal
basis of this plane.  Write $\mu= \alpha x+ \beta v$ with $\alpha
^2+\beta^2=1$. Then, as $x$ and $w$ both belong to $T$ we have $$d(\mu,T)\leq |\mu-(\alpha x+\beta w)|=|\beta|\cdot |v-w| \leq
d(v, T).$$
\end{proof}

\begin{proof}[Proof of Lemma \ref{lem_direction_generique_droit_proj}]
We will work up to a (``projective'') coordinate system of  $\sph^{n-1}$ defined as follows. Let $U_i ^+$ (resp.  $U_i ^-$) denote
$$\{x \in \sph^{n-1} : x_i \geq \epsilon\}$$
(resp. $x_i \leq -\epsilon$) with $\epsilon >0$.
Define then $h_{i} : U_i^+ \rightarrow \R^{n-1}$ (resp. $g_i:U_i^-\to \R^{n-1}$) by
$h_{i}(x_1,\dots,x_n)=(\frac{x_1}{x_i},\dots,\widehat{\frac{x_i}{x_i}} ,\dots,\frac{x_n}{x_i})$ (resp. $g_i(x_1,\dots,x_n)$, here the $\,\widehat{ }\,$ means the term is omitted).
 We can assume that $\bou(l,r)$ entirely fits in one single $U_i^+$ or $U_i^-$, say $U_i^+$ (up to a linear change of coordinate of $\R^n$). 

Through such a   chart, the elements $\sph^{n-1}\cap T$,   $T \in C$, will be identified with their respective images, which are affine subspaces of $\R^{n-1}$.  The set $\sph^{n-1} \cap N_\mu$ also becomes an affine subspace $\Delta$, and
  $\widetilde{\pi}_\mu$ an
orthogonal projection along a line, say $L$. We denote by $\pi$ this projection.  By
Lemma \ref{lem trans aux fibres pi P tilda} and hypothesis (\ref{hypothese key lemma}),  there exists $u >0$
such that for any $T \in C$ (the angle between affine spaces is defined as the angle between the associated vector spaces):
\begin{equation}\label{eq_preuve_angle_key_lemma}
 \angle(L,T) \geq u.
\end{equation}

 We have to find  $\alpha >0$ such  that
for any $P_1,\dots,P_{\kappa}$ in $C$ and any $y \in
\pi(h_i( \bou(l,\frac{r}{2})))$  there exists
$\widehat{\lambda} \in h_i(\bou(l,r)) \cap
\pi^{-1}(y)$ such that:
\begin{equation}\label{eq_preuve_key_lemma_apres_carte}
 d(\, \widehat{\lambda} , \bigcup_{i=1} ^{\kappa} P_i) \geq \alpha.
\end{equation}

For any $y \in \pi(h_i(\bou(l,\frac{r}{2})))$, the length of $\pi^{-1}(y) \cap h_i(\bou(l,r))$ is
bounded below away from zero by a strictly positive real number $\alpha_0$ ($h_i$ is bi-Lipschitz).

   It is an easy exercise of elementary geometry to derive from (\ref{eq_preuve_angle_key_lemma}) that for any $\alpha>0$ the set $\{x \in \pi^{-1}(y):d(x,T)\le \alpha \}$ is a segment of length not greater than $\frac{2 \alpha}{u}$, for all $T\in C$ and $y \in \Delta$.  


Let $\alpha:=\frac{\alpha_0u}{4\kappa}$.
By the preceding paragraph, we see that if (\ref{eq_preuve_key_lemma_apres_carte}) failed for some $y \in
\pi(h_i(\bou(l,\frac{r}{2})))$, we could cover 
$\pi^{-1}(y) \cap h_i(\bou(l,r))$ by $\kappa$ segments of
length not greater than $\frac{\alpha_0}{2\kappa}$. This contradicts the
fact that the length of $\pi^{-1}(y) \cap h_i(\bou(l,r))$ is not less than $\alpha_0$.
\end{proof}

\subsection{Proof of Theorem \ref{thm_proj_reg_hom_pres_familles}}\label{sect_proof_main}
  By Theorem \ref{thm_existence_des_familles}, there is a definable partition $\Pa$ of $\R^m$ such that for every $B \in \Pa$ there exists a  regular system of
hypersurfaces compatible with $A_{B}$. 
 Fix  $B \in \Pa$ and such a regular system of
hypersurfaces $H=\hk$.

 For each $t\in B$, we shall define the desired definable mapping $h_t$
over $E(H_{k,t},\lambda_k)$  by induction on $k$, in such a way that for all $t\in B$
$$h_t(E(H_{k,t},\lambda_k))=E(F_{k,t} ,e_n),$$  where
$F_{k,t}$ is the graph for $e_n$ of  $\eta_{k,t}: \R^{n-1}\to\R$, with $(\eta_{k,t})_{t \in B}$ uniformly Lipschitz definable family of functions.

For $k=1$, choose an orthonormal  basis of $N_{\lambda_1}$ and set
$h_t(q):=(x_{\lambda_1},q_{\lambda_1})$, where $x_{\lambda_1}$
 stands for the coordinates of $\pi_{\lambda_1}(q)$ in this basis.

Let now $k \geq 1$. By $(i)$ of Definition \ref{dfn_systemes}, for any $t\in B$ the sets $H_{k,t}$ and $H_{k+1,t}$ are the respective
graphs for $\lambda_k$ of two functions $\xi_{k,t}$ and
$\xi'_{k,t}$. For $q \in E(H_{k+1,t},\lambda_{k}) \setminus
E(H_{k,t},\lambda_k)$, extend $h_t$ by defining $h_t(q)$ as the element:
$$h_t\big(\pi_{\lambda_k}(q)+\xi_{k,t} (\pi_{\lambda_k}(q))\cdot \lambda_k\big)
+\big(q_{\lambda_k}-\xi_{k,t} ( \pi_{\lambda_k}(q))\big)e_n.$$

Thanks to the property $(ii)$ of  Definition \ref{dfn_systemes}, we have: $$ E(H_{k+1,t},\lambda_{k+1})=
E(H_{k+1,t},\lambda_{k}),$$ and hence  $h_t$ is actually defined over  $
E(H_{k+1,t},\lambda_{k+1})$. Since $\xi_{k,t}$ is $C$-Lipschitz with $C$ independent of $t$, the $h_t$'s constitute a uniformly bi-Lipschitz family of
 homeomorphisms. Note also that the image of $h_t$ is
$E(F_{k+1,t},e_n)$, where $F_{k+1,t}$ is the graph (for $e_n$) of the  uniformly Lipschitz family of
functions:
$$\eta_{k+1,t}(x):=\eta_{k,t} (x)+(\xi'_{k,t}-\xi_{k,t})
\circ \pi_{\lambda_k}\circ h^{-1}_t(x,\eta_{k,t}(x)),$$
for $(t,x) \in  B \times \R^{n-1}$. This completes the induction step, giving $h_t$ over $E(H_{b,t} , \lambda_b)$. To extend $h_t$ to the
whole of $ \R^n$ do it similarly as in the case $k=1$.

\subsection{Regular vectors and set germs}\label{sect regular vectors and set germs}
For $R$ positive real number and $n>1$  we set   $$\ccn :=\{(t,x) \in [0,+\infty)\times \R^{n-1}: |x|\leq Rt\}. $$ 
We also set $\mathcal{C}_1(R):=[0,+\infty)$. 

Our purpose is to show Theorem \ref{thm_proj_loc}, which is an improvement of Corollary \ref{cor_proj_reg_hom_pres}  asserting that, in the case of germs of subsets of $\mathcal{C}_{n}(R)$, 
 the provided homeomorphism may be  required to preserve the first coordinate in the canonical basis. This fact is an essential ingredient of the Lipschitz conic structure theorem \cite{gvpoincare,livre}, which recently proved very useful to study Sobolev spaces of bounded subanalytic domains  \cite{ poincwirt, linfty, gvpoincare, trace,lprime}.

\begin{dfn}
Let $A, B \subset \R^n$. A definable map $h:A \to B$  is {\bf vertical}\index{vertical} if it
preserves the first coordinate in the canonical basis of $\R^n$, i.e. if for any $t \in \R$, $\pi(h(t,x))=t$, for all $x \in A_t$,  where $\pi:\R^n \to \R$ is the orthogonal projection onto the first coordinate.
\end{dfn}

  We start with a preliminary lemma which is of its own interest. We use the notation $f(t)\ll g(t)$ to express that $f(t)\le g(t)\phi(t)$, for some function $\phi$ tending to zero as $t$ goes to zero.

\begin{lem}\label{lem_lipschitz_param_local}
Let $h:(X,0)\to( \mathcal{C}_n(R),0)$ be a germ of vertical definable  map, with $X\subset \mathcal{C}_n(R')$, for some $R$ and $R'$. If $(h_t)_{t \in \R}$ is uniformly Lipschitz  then $h$ is a Lipschitz map germ.
\end{lem}
\begin{proof}
Suppose that  $h$ fails to be Lipschitz.
 Then, by Curve Selection Lemma,  we can find two definable arcs in $X$,
say $p(t)$ and $q(t)$, tending to zero along which:
\begin{equation}\label{eq hyp p et q }
|p(t)-q(t)| \ll |h(p(t))-h(q(t))|.
\end{equation}

 We may assume that  $p(t)$ (and so $h(p(t))$) is parametrized by its first coordinate (since the first coordinate of  $p(t)$ induces a homeomorphism from a right-hand-side neighborhood of zero in $\R$ onto a right-hand-side neighborhood of zero in $\R$), i.e. we may assume $p(t)=(t,p_2(t),\dots,p_n(t))$.

 As  $p(t)$ and $h(p(t))$ are definable arcs in $\C_n(R')$ and $\C_n(R)$ respectively, we
have:
\begin{equation}\label{eq proof r lips along h(p)}
|h(p(t))-h(p(t'))| \sim |t-t'|\leq |p(t)-q(t)|
\end{equation} and
\begin{equation}\label{eq proof r lips along p}
|p(t)-p(t')| \sim |t-t'| \leq |p(t)-q(t)|,
\end{equation}
where $t'$ denotes the first coordinate  of $q(t)$.

Therefore, we can easily derive from  (\ref{eq hyp p et q }), (\ref{eq proof r lips
along h(p)}), and (\ref{eq proof r lips along p}): 
$$|h(p(t))-h(q(t))| \sim |h(p(t'))-h(q(t))| \sim |p(t')-q(t)|\lesssim |p(t)-q(t)|,$$
%
 in contradiction with (\ref{eq hyp p et q }). 
\end{proof}

\begin{thm}\label{thm_proj_loc}
Let $X$ be the germ at $0$ of a definable  subset of $\mathcal{C}_{n}(R)$ (for some $R$) of
empty interior. There exists a germ of
vertical bi-Lipschitz definable homeomorphism (onto its image) $H:(\C_{n}(R),0)\to
(\C_n(R),0)$ such that $e_{n}$ is regular for $H(X)$.
\end{thm}
\begin{proof}
We denote by $e_1,\dots,e_n$ the canonical basis of $\R^{n}$ and by $e_1',\dots,e_{n-1}'$   the canonical basis of $\R^{n-1}$ (so that $e_n=(0,e'_{n-1})$).
 Apply Theorem \ref{thm_proj_reg_hom_pres_familles} to $X$, regarded as a family of $\R\times \R^{n-1}$.
This provides a uniformly bi-Lipschitz family of homeomorphisms $h_t:\R^{n-1} \rightarrow
\R^{n-1}$, $t\in (0,\ep)$,  such that $e'_{n-1}$ is regular for the family 
$(h_t(X_t))_{t \in \R}$.   Up to a family of translations, we may
assume that $h(t,0)\equiv 0$, which implies that $H:(\R^n,0)\to (\R^n ,0)$, $(t,x)\mapsto (t,h_t(x))$, maps $\C_{n}(R)$ into
$\C_n(R')$ for some $R'$ (and up to a family of homothetic transformations we may  assume $R=R'$). 
By Lemma \ref{lem_lipschitz_param_local}, the map-germ $H$  is bi-Lipschitz near the origin.   

 We now have to check
 that $e_n$ is regular for  the germ of the definable set  $Y:=H(X)$. Suppose not. Then, by Curve Selection Lemma,  there exists
 a definable arc $\gamma : [0,\epsilon]\rightarrow Y_{reg}$
 with $\gamma(0)=0$ and  $e_n \in \tau :=\lim_{t\to 0} T_{\gamma(t)}
 Y_{reg}$. But, as $e'_{n-1}$ is regular for the family
 $(Y_{t})_{t\in [0,\ep]}$, we have $e'_{n-1}\notin \lim_{t \to 0} T_{\tilde{\gamma}(t)}
 Y_{\gamma_{1}(t)}$, if $\gamma(t)=(\gamma_1(t),\tilde{\gamma}(t))\in \R \times \R^{n-1}$. This implies that: $$\tau \cap N_{e_1} \neq \lim_{t\to 0}\, (T_{\gamma(t)}
 Y_{reg} \cap N_{e_1})$$
(since the latter does not contains the vector  $e_{n}=(0,e'_{n-1})$ while the former does). Hence, $\tau$ cannot be transverse to  $ N_{e_1}
 $ (since otherwise the intersection with the limit would be the limit of the intersection)  which means that  $e_{1}$ is orthogonal to  $\tau $. This implies that the limit
 vector  $\lim_{t\to 0}
 \frac{\gamma(t)}{|\gamma(t)|} =\lim_{t\to 0}
 \frac{\gamma'(t)}{|\gamma'(t)|}\in\tau $ is orthogonal to $e_{1}$, from which we can conclude
  $$\lim_{t \to 0} \frac {\gamma_{1}(t)}{|\gamma(t)|}=0,$$
  in contradiction with  $\gamma(t) \in \C_n(R)$.
\end{proof}
\end{section}

\begin{section}{Definable bi-Lipschitz triviality in polynomially bounded o-minimal structures}\label{sect_triviality}
The results of this section are valid {\it under the extra assumption that the structure is polynomially bounded}. It is not difficult to produce counterexamples to these results (except however to Proposition \ref{pro_lipschitz_parametres}) as soon as this assumption fails. We are going to establish  a bi-Lipschitz triviality theorem for definable families (Theorem \ref{hardt}), from which we will derive a stratification result (Corollary \ref{cor_strat_bil_triv}).

We denote by $\mathscr{F} $ the valuation field of the structure, which is the subfield of $\R$ constituted by all the real numbers $\alpha$ for which the function $(0,+\infty)\ni x\mapsto x^\alpha\in \R$ is definable.

\begin{dfn}\label{dfn_trivialite_bilip}
Let $A\in \s_{m+n}$. We will say
that $A$ is \textbf{definably bi-Lipschitz trivial along} $U
\subset \R^m$ if it is definably topologically trivial along this set, with a trivialization
  $h_t:A_{t_0} \to A_t$ which is bi-Lipschitz  for every $t\in U$.
\end{dfn}

We shall show that this happens for generic parameters (Theorem \ref{hardt}).  We first recall a result that we shall need (restated with the notations of the present article), sometimes called ``the preparation theorem for definable functions''.
\begin{thm}\label{thm_prep}\cite[Theorem $2.1$]{vds}
Given some definable functions $f_1,\dots,f_l:\R^{n} \to \R$, there is a
definable partition $\C$ of $\R^{n}$  such that for each set $S \in \C$, there
are exponents $\alpha_1 ,\dots, \alpha_l \in   \mathscr{F}$,  as well as definable functions $\theta, a_1 ,\dots, a_l :\R^{n-1}\to \R$, satisfying $\Gamma_\theta\cap S=\emptyset$ and for $x=(\xt,x_n)\in S \subset \R^{n-1}\times \R$ and $i\in \{1,\dots, l\}$:  $$f_i(x) \sim |x_n-\theta(\xt)|^{\alpha_i}a_i(\xt).$$
\end{thm}
This theorem will be needed to establish the following lemma.

\begin{lem}\label{lem_eq_dist_cor_th_prep}
Let $\xi :  \R^{m+n} \rightarrow \R$ be a definable nonnegative function.
There exist some definable subsets of $\R^{m+n}$, say $W_1,\dots,W_k$,  and a  definable  partition $\Pa$ of $\R^{m+n}$ such that for any $V \in \Pa$ there are  $\alpha_1,\dots,\alpha_k$ in $\mathscr{F} $ such that for each $\tim$ we have on $V_t \subset \R^n$:  \begin{equation}\label{eq_equiv_dist}\xi_t(x) \sim d(x,W_{1,t})^{\alpha_1} \cdots d(x,W_{k,t})^{\alpha_k}   .
      \end{equation}

\end{lem}

\begin{proof}
We prove it by induction on $n$. For $n=1$, the result follows from
 Theorem \ref{thm_prep}. 
 Let  $n\geq 2 $ and assume that the proposition is true for $(n-1)$. Let  $\lambda_1,\dots,
\lambda_N$ be the elements of $\sph^{n-1}$ given by Proposition \ref{pro_dec_L_regulier_avec_proj_fixees}.

 For each $i$, applying the Preparation Theorem (Theorem \ref{thm_prep})  to $\xi \circ \Lambda_i: \R^{m+n} \to \mathbb{R}$, where $\Lambda_i$ is an orthogonal linear mapping of $\R^{m+n}$ sending $(0_{\R^m}, e_n)$ onto $(0_{\R^m},\lambda_i)$ and preserving the $m$ first coordinates (below, we sometimes regard $\Lambda_i$ as a transformation of $\R^n$),   we obtain a partition of   $\R^{m+n}$. The images of all the elements of this partition under the map $\Lambda_i$ provide a new partition of $\R^{m+n}$, denoted by $\Pa_i$. Let $ (V_j)_{j \in J}$ be a common refinement of the  $\Pa_i$'s. 
Applying
Proposition \ref{pro_dec_L_regulier_avec_proj_fixees} to the finite family
constituted by all the sets of the partition  $(V_j)_{j \in J}$, we get a partition $\Sigma$ of $\R^{m+n}$ into cells.

Let $E \in \Sigma$ be an open cell.  By construction and Proposition \ref{pro_dec_L_regulier_avec_proj_fixees}, there  is $i\le N$  such that $\lambda_i$ is regular for $\delta E $.  It means that $e_n$ is regular for the family $(\Lambda_i^{-1}(\delta E)_t)_{\tim}$.  Hence, it follows from  Proposition \ref{pro_proj_reg_decomposition_en_graphes} that there is a partition $\mathcal{Q}_E$  of $\Lambda_i^{-1}(cl(E))$ into cells, such that each element $C$ is either the graph of a uniformly Lipschitz family of functions  or a set of the form
\begin{equation}\label{eq_cell_sim_dist}
C = \{(z,y)\in B\times \R : \eta_{1}(z) < y < \eta_{2}(z)\},
\end{equation}
with $B\in \s_{m+n-1}$ and 
 $\eta_1< \eta_2$   definable functions on $B$ such that $(\eta_{1,t})_{\tim}$ and  $(\eta_{2,t})_{\tim}$ are uniformly Lipschitz.
 
Observe that it suffices to show the desired statement for the  restriction to each cell  $C\in \mathcal{Q}_E$ of the family of  functions $\xi_t\circ \Lambda_i$, $\tim$. For the elements $C$ of the partition $\mathcal{Q}_E$ which are graphs of some uniformly Lipschitz family of functions, one
may easily  deduce the result from the induction hypothesis.

Fix thus a cell $C\subset \Lambda_i^{-1}(E)$ as in (\ref{eq_cell_sim_dist}).  There is $j$ such that $C\subset \Lambda_i^{-1}(V_j)$. By construction,
there are $r \in \mathscr{F} $ and some functions $a$ and $\theta$ on the basis $B$ of $C$ such that for $x=(\xt,x_n)\in C_t$, $\tim$, we have:$$\xi_t\circ \Lambda_i(x)\sim a_t(\xt)|x_n-\theta_t(\xt)|^r.$$    
 Thanks to the induction hypothesis we thus only have to check the result for the function $|x_n-\theta_t(\xt)|$.

 As $\Gamma_\theta\cap C=\emptyset$,
 we  can assume that for every $(t,\xt)\in B$,   either
$\theta_t(\xt) \leq \eta_{1,t}(\xt)$ or  $ \theta_t(\xt) \geq \eta_{2,t}(\xt)$.
Assume for instance that $\theta_t(\xt) \leq \eta_{1,t}(\xt)$. 
Writing for $t \in\supp_m(C)$ and $x=(\xt,x_n) \in C_{t}$: $$x_n -
\theta_t(\xt)=(x_n-\eta_{1,t}(\xt))+(\eta_{1,t}(\xt)-\theta_t(\xt)),$$ we see that (up to a partition of $C$ we may assume that the terms of the right-hand-side are comparable for the partial order relation $\le$)
$|x_n - \theta_t(\xt)|$
 is $\sim$ either to $|x_n-\eta_{1,t}(\xt)|$ or to
 $|\eta_{1,t}(\xt)-\theta_t(\xt)|$.

 For the latter functions, since they are $(n-1)$-variable functions, the desired result is a consequence of the induction
  hypothesis. Moreover, since
 $\eta_{1,t}$ is Lipschitz,  $|x_n -
 \eta_{1,t}(\xt)|$ is $\sim$ to the distance to the graph of $\eta_{1,t}$
  for every $t$. This shows the result for the given  cell $C$.
\end{proof}

\begin{rem} The constants of the equivalence in the above lemma depend on $t$. However, the family of exponents $\alpha_1,\dots,\alpha_k$  just depends on $V\in \Pa$.
\end{rem}
We recall that the structure is assumed to be polynomially bounded in this section. 
\begin{thm}\label{hardt}
Given $A\in \s_{m+n}$, there
exists a definable  partition of $\R^m$ such that  $A$
is definably bi-Lipschitz trivial along each  element of this
partition.
\end{thm}
\begin{proof}
We prove the result by induction on $n$. We shall show that the trivialization $H$ may be required to induce a trivialization of some given definable subsets of $A$. 

Let $A\in \s_{m+n}$ and let $C_1,\dots,C_k$ be some definable subsets of $A$. Apply Theorem \ref{thm_proj_reg_hom_pres_familles} to the set $\{(t,x):x \in \delta A_t\cup \bigcup_{i=1}^k \delta C_{i,t} \}$. This provides  a definable family of   bi-Lipschitz maps $G_t:\R^n \to \R^n$, $\tim$, such that $e_n$ is regular for the families of sets   $(\delta G_t(C_{i,t}))_{\tim}$, $i=1,\dots,k$, and  $(\delta G_t(A_t))_{\tim}$. 

 As we can work up to a family of bi-Lipschitz maps, we will identify $G_t$ with the identity map. By Propositions \ref{pro_extension_fonction_lipschitz}, \ref{pro_proj_reg_decomposition_en_graphes}, and \ref{pro_famille_lipschitz_ordonnees}, we can find some  definable  functions $\xi_{1}\le \dots\le \xi_{s}$ on $ \R^{m+n-1}$, with $(\xi_{i,t})_{\tim}$ uniformly Lipschitz for all $i$, and a cell decomposition $\D$ of $\R^{m+n-1}$ such that $A$ and the $C_j$'s are unions of some graphs $\Gamma _{\xi_{i|D}}$, $i\in \{1,\dots,s\}$, $D \in \D$, or bands $(\xi_{i|D},\xi_{i+1|D})$, $i\in \{0,\dots, s\}$, $D \in \D$ (where $\xi_0\equiv -\infty$ and $\xi_{s+1}\equiv+\infty$).

Refining the cell decomposition $\D$ if necessary (without changing notations), we can assume it to be compatible with the zero loci of the functions $(\xi_{i+1}-\xi_i)$. By Lemma \ref{lem_eq_dist_cor_th_prep}, up to an extra refinement of the cell decomposition, we can assume that there are finitely many definable subsets $W_1,\dots,W_c$ of $\R^m \times \R^{n-1}$ such that on every cell we can find  $r_1,\dots,r_c$ in $\mathscr{F} $ such that for all $i=1,\dots,s-1$ and any $\tim$:
\begin{equation}\label{eq_W_i_preuve_hardt}\xi_{i+1,t}(x)-\xi_{i,t}(x) \sim d(x,W_{1,t})^{r_1}\cdots d(x,W_{c,t})^{r_c}.\end{equation}
Refining one more time the cell decomposition $\D$, we may assume that the $W_i$'s are unions of cells. 

Applying now the induction hypothesis to the cells  of $\D$ provides a partition $\Pa$.  Fix $B \in \Pa$
and let $H(t,x)=(t,h_t(x))$ denote the obtained trivialization of $B \times \R^{n-1}$ along $B$.  We have $h_t(C_{t_0})= C_t$ for some $t_0\in B$ and for all $C\in\D$. We are going to lift the isotopy $H$ to an isotopy of $B\times \R^n$.

Given a point $(t,x)\in B \times \R^{n-1}$ and $1\le i \le s-1$ let  $$\widetilde{H}(t,x,\nu \xi_{i,t_0}(x)+(1-\nu)\xi_{i+1,t_0}(x)):=(t,h_t(x),\nu\xi_{i,t}(h_t(x))+(1-\nu)\xi_{i+1,t}(h_t(x)),$$
for all $\nu \in [0,1]$. Set also for $\nu \in (0,\infty)$:
$$\widetilde{H}(t,x,\xi_{1,t_0}(x)-\nu):= (t,h_t(x),\xi_{1,t}(h_t(x))-\nu),$$
as well as
$$\widetilde{H}(t,x,\xi_{s,t_0}(x)+\nu):= (t,h_t(x),\xi_{s,t}(h_t(x))+\nu).$$
 Because $\D$ is compatible with the zero loci of the functions $(\xi_{i+1}-\xi_i)$  and since the trivialization $h$ was required to preserve the cells of $\D$, it is easily seen that $\widetilde{H}_t$ is a continuous mapping for each $\tim$. Observe also that, since the $W_i$'s  are unions of cells of $\D$, we have  $h_t(W_{i,t_0})=W_{i,t}$, for all $i$. Since $h_t$ is bi-Lipschitz for every $t\in B$, we can derive from (\ref{eq_W_i_preuve_hardt}), that for each $t \in B$ we have:
$$(\xi_{i+1,t}-\xi_{i,t})\circ h_t\sim (\xi_{i+1,t_0}-\xi_{i,t_0}).$$
This shows the bi-Lipschitzness of $\widetilde{H}_t$ on the sets $[\xi_{i,t|D_t},\xi_{i+1,t|D_t}]$, $D \in \D$, $D \subset B$, $i<s$. The bi-Lipschitzness of $\widetilde{H}_t$ on  the sets $(-\infty, \xi_{1,t|D_t})$ and  $(\xi_{s,t|D_t},+\infty)$ is clear since the  $(\xi_{i,t})_{t\in B}$  are families of  Lipschitz functions.

 The continuity of $H_t$ and $H_t^{-1}$ with respect to $t$ follows from a well-known fact, up to an extra refinement of the partition of the parameter space \cite[Lemma 5.17 and Exercise 5.21]{costeomin}.
  \end{proof}

\begin{rem}\label{rem thm princ}
   We have proved a stronger statement since the
isotopy is also  defined on the ambient space $B\times \R^n $. We can also
require the isotopy to preserve a finite number of given definable
 subfamilies of $A$.
\end{rem}

In Theorem \ref{hardt},  the constructed trivialization $h_t$ is Lipschitz for every $t$ (see Definition \ref{dfn_trivialite_bilip}).  The Lipschitz condition  may also be required to hold with respect
to the parameter $t$ on relatively compact sets, as it  will be established by  Proposition
\ref{pro_hardt_avec_parametres}, which requires the following proposition.

\begin{pro}\label{pro_lipschitz_parametres}
 Let $A \in \s_{m+n}$ and let $f_t:A_t \to \R$ be a definable family of functions. If $f_t$ is Lipschitz for all $\tim$ then there exists a definable partition $\Pa$ of $\R^m$ such that for every $B \in \Pa$,   $f:A\to  \R$, $(t,x)\mapsto f_t(x)$ induces a Lipschitz function on  $A\cap K$, for every compact subset $K$ of  $ B\times \R^n$.
\end{pro}
\begin{proof} We prove the result by induction on $m$. The case $m=0$ being vacuous, assume the result to be true for $(m-1)$, $m\ge 1$. By Proposition \ref{pro_extension_fonction_lipschitz} (see Remark \ref{rem_extension_familles_fonctions_lipschitz}), we may assume that $A= \R^{m+n}$.  It is well-known that there is  a definable partition $\Pa$ of the parameter space, such that  $f$ is continuous on every $B\times \R^n$, $B\in \Pa$ (again, see  \cite[Lemma 5.17 and Exercise 5.21]{costeomin}). Fix an element $B\in \Pa$ (we shall refine several times the partition $\Pa$). 

 We start with the (easier) case where $\dim B <m$. In this case, there is a partition of $B$ such that every element of this partition has a regular vector (using for instance Remark \ref{rem_alpha_flat}), that, without loss of generality, we can assume to be $e_m\in \sph^{m-1}$.   Thanks to Proposition \ref{pro_proj_reg_decomposition_en_graphes}, it is therefore enough to deal with the case where $B$ is the graph of a Lipschitz function, say $\xi:D \to \R$, $D \in \s_{m-1}$. The result in this case now follows from the induction hypothesis applied to the function $D \times \R^n \ni (t,x) \mapsto f(t,\xi(t),x)$.

 We now address the case  $\dim B=m$. The function $B\ni t \mapsto L_{f_t}$ being definable, partitioning $B$ if necessary, we  can assume this function to be continuous on this set. In particular, it is bounded on compact subsets of $B$.
 Let $Z$ be the set of points  $q \in \Gamma_f$ for which there exists a sequence $q_k\in (\Gamma_f)_{reg}$ tending to $q$ such that   $$(0_{\R^m},e_{n+1}) \in\lim T_{q_k} (\Gamma_f)_{reg},$$ where $e_{n+1}$ is the last vector of the canonical basis of $\R^{n+1}$. Let $\pi :\R^m \times \R^{n+1} \to \R^m$ denote the projection omitting the last $(n+1)$ coordinates. We claim that $\pi(Z)$ has dimension less than $m$. 

 Assume otherwise. Take a $(w)$-regular stratification of $\Gamma_f$ compatible with $Z$   and  let $S\subset Z$  be a stratum such that $\pi(S)$ has dimension $m$. Let $S'$ be the set of points of $S$ at which $\pi_{|S}$ is a submersion. Since $\pi(S)$ is of dimension $m$, by Sard's Theorem,  the set $S'$ cannot be empty. Moreover, by definition of $S'$, $T_q S'$ is transverse to $\{0_{\R^m}\}\times \R^{n+1}$   at any point $q$ of $S'$.

Let $q \in S'\subset Z$. By definition of $Z$,  there is a sequence $q_k$ tending to $q$ such that $(0_{\R^m},e_{n+1}) \in \tau_q:=\lim T_{q_k} (\Gamma_f)_{reg}$.  The $(w)$ condition ensures that $\tau_q\supset T_q S'$ ($S'$ is a manifold for it is open in $S$). Consequently, $\tau_q$ is transverse to $\{0_{\R^m}\}\times \R^{n+1}$ as well.

 But since $L_{f_t}$ is locally bounded (it was assumed to be continuous),  the vector $e_{n+1}$ does not belong to $\lim T_{x_k}\Gamma_{f_{t_k}}$, if $q_k=(t_k,x_k)$ in $\R^m \times \R^{n+1}$, which means that 
 $$(\lim T_{q_k}\Gamma_f) \cap\{0_{\R^m}\}\times \R^{n+1}\neq \lim  \big{(}T_{q_k}\Gamma_f \cap\{0_{\R^m}\}\times \R^{n+1}\big{)} $$
(since the latter  does not contain the vector  $(0_{\R^m},e_{n+1})$ while the former does), and hence, that $\tau_q$ cannot be transverse to $\{0_{\R^m}\}\times \R^{n+1}$ (since otherwise the intersection with the limit would be the limit of the intersection). A contradiction.

This establishes that $\dim \pi(Z)<m$. Since we can refine $\Pa$ into a partition which is compatible with $\pi(Z)$, we thus see that we can suppose $B \subset\R^m \setminus \pi(Z)$  (we are dealing with the case $\dim B=m$).

  For $(t,R)\in (\R^m \setminus \pi(Z)) \times [0,+\infty)$ set:   $$\varphi(t,R):=\sup\{\frac{\pa f}{\pa t}(t,x):x\in \bou(0_{\R^n},R), \mbox{ $f$ is $\cc^1$ at $x$}   \}$$  (which is finite, by definition of $Z$, since $L_{f_t}$ is bounded).
  As $\varphi$ is definable,  up to a partition of $B$, this function  may be assumed to be continuous (and thus bounded on compact sets)  for $R \ge \zeta(t)$, with $\zeta:B\to \R$ definable function.   The function $f$ therefore induces a function which is Lipschitz with respect to the inner metric on every compact set of $B\times \R^n$.
 By Theorem \ref{thm_kp}, up to an extra refinement partition,  we can suppose that the inner  metric and the outer metric of $B$ are equivalent, which means that so are the inner and outer metrics of $B\times \R^n$, establishing that $f$ is Lipschitz on every compact set of $B\times \R^n$. 
\end{proof}

As a matter of fact, the trivialization given by Theorem \ref{hardt} may be required to  satisfy the Lipschitz condition with respect to the  parameters on compact  sets:

\begin{pro}\label{pro_hardt_avec_parametres}
Let $A \in \s_{m+n}$. Refining the partition   provided by 
 Theorem \ref{hardt}, we may obtain the following extra fact: for any element $B$ of this partition, the trivialization $H:B\times A_{t_0} \to A_{B} $, $(t,x) \mapsto (t,h_t(x))$, induces a bi-Lipschitz mapping on  $ (B\times A_{t_0})\cap K$, for every compact subset $K$ of  $ B\times \R^n$.
\end{pro}
\begin{proof}
 This is a consequence of Theorem \ref{hardt} and Proposition \ref{pro_lipschitz_parametres}.
\end{proof}

The compactness assumption is essential, as shown by the following example.
\begin{exa}\label{ex lips param a  l infini}
Consider the  set  $A=\{(t,x,y) \in  \R^3 :
y=tx\}$. By Theorem \ref{hardt}, this set is bi-Lipschitz trivial along a right-hand-side neighborhood of zero in $\R$. However, it is easy to check that we could not
 require a trivialization $H(t,x,y)$ to be  bi-Lipschitz   with
respect the parameter $t$, even along a compact interval (i.e., we have to require that $x$ and $y$ also remain in a compact set in order to ensure Lipschitzness with respect to $t$).
\end{exa}

The inconvenience of bi-Lipschitz triviality theorems that are provided by integration of Lipschitz vector fields, such as the bi-Lipschitz version of Thom-Mather isotopy theorem that holds on Mostowski's Lipschitz stratifications \cite{most, parusub}, is that they do not provide definable trivializations. Theorem \ref{hardt} enables us to construct stratifications that are definably bi-Lipschitz trivial along the strata. 

In the definition below, we  write smooth without specifying the degree of smoothness. It is well-known that one can construct $\cc^k$ stratifications, for every given $k$. When the structure has $\cc^\infty$ cell decomposition, we can construct stratifications that have $\cc^\infty$ strata. By ``smooth retraction'', we mean a retraction that extends to a smooth mapping on a neighborhood of $V_S$ in $\R^n$.  Definable $\cc^k$ manifolds admit definable $\cc^{k-1}$ tubular neighborhoods \cite{costeomin}.
\begin{dfn}\label{dfn_strat}
 A  stratification $\Sigma$ of a set $X$ is {\bf locally definably bi-Lipschitz trivial}\index{locally definably bi-Lipschitz trivial stratification} if for every $S\in \Sigma$, there are an open neighborhood $V_S$ of $S$ in $X$ and a smooth definable retraction $\pi_S:V_S\to S$  such that every $x_0\in S$ has an open neighborhood $W$ in $S$ for which there is a definable bi-Lipschitz homeomorphism  $$\Lambda:\pi_S^{-1}(W)\to \pi_S^{-1}(x_0) \times W, $$ satisfying:
  \begin{enumerate}[(i)]
   \item  $\pi_S(\Lambda^{-1}(x,y))= y$, for all $(x,y)\in \pi_S^{-1}(x_0)\times  W$.
   \item   $\Sigma_{x_0}:=\{  \pi_S^{-1}(x_0)\cap Y:Y\in \Sigma\} $ is a stratification of $ \pi_S^{-1}(x_0)$,  and $\Lambda(\pi_S^{-1}(W)\cap Y)=(\pi_S^{-1}(x_0)\cap Y)\times W$, for all $Y\in \Sigma$.
  \end{enumerate}
\end{dfn}
 
We now can draw the following consequence of our definable bi-Lipschitz triviality theorem, valuable for applications \cite{trace,lprime}:
\begin{cor}\label{cor_strat_bil_triv}
 Given a definable set $X$, we can find a stratification of this set which is locally definably bi-Lipschitz trivial. This stratification may be required to be compatible with finitely many given definable subsets of $X$.
\end{cor}
\begin{proof}This follows from standard arguments of construction of stratifications.
 Theorem \ref{hardt} and Proposition \ref{pro_lipschitz_parametres} yield that local definable bi-Lipschitz triviality holds generically, which is sometimes rephrased by saying that it is a stratifying condition (see for instance \cite[Proposition 2.7.5]{livre} for more details). We can require our stratification to satisfy Whitney's $(a)$ condition (which is also a stratifying condition \cite{bcr, taleloi, livre}), which yields that $\Sigma_{x_0}$ (in $(ii)$ of Definition \ref{dfn_strat}) exclusively consists of manifolds.
\end{proof}

\end{section}

\end{document}